\documentclass[12pt]{article}
\usepackage{amsfonts, amssymb, epsfig,subfigure}
\usepackage{mathrsfs}
\usepackage{stmaryrd}
\usepackage{graphicx,amsmath}
\usepackage{latexsym}
\usepackage{verbatim}
\usepackage{slashed}
\usepackage{color}
\makeatletter
\@addtoreset{equation}{section}

\begin{document}

\newcommand{\E}{\mathbb{E}}
\newcommand{\PP}{\mathbb{P}}
\newcommand{\CP}{\mathcal{P}}
\newcommand{\CU}{\mathcal{U}}
\newcommand{\CW}{\mathcal{W}}
\newcommand{\CK}{\mathcal{K}}
\newcommand{\RR}{\mathbb{R}}
\newcommand{\LL}{\mathbb{L}}
\newcommand{\HH}{\mathbb{H}}
\newcommand{\CL}{\mathcal{L}}
\newcommand{\NN}{\mathbb{N}}
\newcommand{\CC}{\mathcal{C}}
\newcommand{\BB}{\mathcal{B}}
\newcommand{\SM}{\mathbb{S}}
\newcommand{\CM}{\mathcal{M}}
\newcommand{\II}{\mathbf{1}}
\newcommand{\hX}{\hat X}
\newcommand{\bq}{\bar q}
\newcommand{\bV}{\bar V}

\newtheorem{thm}{Theorem}[section]
\newtheorem{prop}[thm]{Proposition}
\newtheorem{assp}[thm]{Assumption}
\newtheorem{lem}[thm]{Lemma}
\newtheorem{rem}[thm]{Remark}
\newtheorem{cor}[thm]{Corollary}
\newtheorem{De}[thm]{Definition}
\newtheorem{expl}[thm]{Example}
\newtheorem{ca}[thm]{Case}
\newcommand{\thmref}[1]{Theorem~{\rm \ref{#1}}}
\newcommand{\lemref}[1]{Lemma~{\rm \ref{#1}}}
\newcommand{\corref}[1]{Corollary~{\rm \ref{#1}}}
\newcommand{\propref}[1]{Proposition~{\rm \ref{#1}}}
\newcommand{\defref}[1]{Definition~{\rm \ref{#1}}}
\newcommand{\remref}[1]{Remark~{\rm \ref{#1}}}
\newcommand{\exmref}[1]{Example~{\rm \ref{#1}}}
\newcommand{\lelemref}[1]{Lemma~{\bf \ref{#1}}}

\newcommand\tq{{\scriptstyle{3\over 4 }\scriptstyle}}
\newcommand\qua{{\scriptstyle{1\over 4 }\scriptstyle}}
\newcommand\hf{{\textstyle{1\over 2 }\displaystyle}}
\newcommand\hhf{{\scriptstyle{1\over 2 }\scriptstyle}}

\newcommand{\proof}{\noindent {\it Proof}. }
\newcommand{\eproof}{\hfill $\Box$} 

\def\a{\alpha} \def\g{\gamma} \def\nn{\nonumber}
\def\e{\varepsilon} \def\z{\zeta} \def\y{\eta} \def\o{\theta}
\def\vo{\vartheta} \def\k{\kappa} \def\lbd{\lambda} \def\m{\mu} \def\n{\nu}
\def\x{\xi}  \def\r{\rho} \def\s{\sigma}
\def\p{\phi} \def\f{\varphi}   \def\w{\omega}
\def\q{\surd} \def\i{\bot} \def\h{\forall} \def\j{\emptyset}

\def\be{\beta} \def\de{\delta} \def\up{\upsilon} \def\eq{\equiv}
\def\ve{\vee} \def\we{\wedge}

\def\d{\mathrm{d}}
\def\F{{\cal F}}
\def\T{\tau} \def\G{\Gamma}  \def\D{\Delta} \def\O{\Theta} \def\L{\Lambda}
\def\X{\Xi} \def\S{\Sigma} \def\W{\Omega}
\def\M{\partial} \def\N{\nabla} \def\Ex{\exists} \def\K{\times}
\def\V{\bigvee} \def\U{\bigwedge}

\def\1{\oslash} \def\2{\oplus} \def\3{\otimes} \def\4{\ominus}
\def\5{\circ} \def\6{\odot} \def\7{\backslash} \def\8{\infty}
\def\9{\bigcap} \def\0{\bigcup} \def\+{\pm} \def\-{\mp}
\def\la{\langle} \def\ra{\rangle}

\def\tl{\tilde}
\def\trace{\hbox{\rm trace}}
\def\diag{\hbox{\rm diag}}
\def\for{\quad\hbox{for }}
\def\refer{\hangindent=0.3in\hangafter=1}

\newcommand\wD{\widehat{\D}}

\title{
\bf
The ergodicity of nonlinear McKean-Vlasov stochastic differential equations with common noise
 }

\author{
{\bf
Xing Chen ${}^{1}$,  
Xiaoyue Li ${}^{1}$,  
Chenggui Yuan ${}^{2}$
 }
\\
${}^1$ School of Mathematical Sciences,\\ Tiangong University, Tianjin, 300387, China.\\
${}^2$ Department of Mathematics,\\ Swansea University, Bay Campus, Swansea, SA1 8EN, UK.
}

\date{}

\maketitle

\begin{abstract}
This paper focuses on the ergodicity of McKean-Vlasov (MV) stochastic differential equations  (SDEs) with common noise (wCN), where coefficients depend on both the state and the measure. A major challenge in this setting is that the underlying Markov operator loses the semigroup property, precluding standard ergodic analyses. To circumvent this issue, we lift the system by considering the joint flow of the solution and its conditional distribution. We first construct a semigroup associated with the measure pair of the solution and its conditional distribution. Under polynomial growth conditions, we prove the existence and uniqueness of the invariant measure for the lifted system by a coupling method and obtain an explicit exponential convergence rate. We subsequently derive the strong law of large numbers by a decoupling approach. Building on these results, we establish the uniform-in-time propagation of chaos for the associated mean-field interacting particle system. Furthermore, we establish the convergence of the distribution of a single particle and the empirical measure of the particle system to the marginals of the invariant measure. Finally, illustrative examples are provided to verify the theoretical findings.


\end{abstract}

\section{Introduction}\label{sec:intr}
McKean-Vlasov (MV) stochastic differential equations  (SDEs) with common noise (wCN) describe an important class of stochastic systems whose dynamics depend on both the microscopic state and the macroscopic conditional distribution. Specifically, we consider the following equation:
\begin{align}\label{eq1.1} \mathrm{d}X_t=f(X_t,\mathcal{L}^1(X_t))\mathrm{d}t+g(X_t,\mathcal{L}^1(X_t))\mathrm{d}B_t+
	g^0(X_t,\mathcal{L}^1(X_t))\mathrm{d}B^0_t, \quad t\geq 0.
\end{align}
where $f:\mathbb{R}^d\times \mathcal{P}(\mathbb{R}^d)\rightarrow \mathbb{R}^d$, $g:\mathbb{R}^d\times \mathcal{P}(\mathbb{R}^d)\rightarrow \mathbb{R}^{d\times d}$ and $g^0: \mathbb{R}^d\times \mathcal{P}(\mathbb{R}^d)\rightarrow \mathbb{R}^{d\times d}$ are Borel measurable functions, $(B_t)_{t\geq0}$ and $(B_t^0)_{t\geq0}$ are two independent $d$-dimensional Brownian motions with $(B_t^0)_{t\geq0}$ representing the common noise. Here for any $t\geq 0$, $\mathcal{L}^1(X_t)$ denotes the conditional distribution of $X_t$ given $\mathcal{F}^0_t$ (see Section \ref{Sect.1} for the details). When $g^0\equiv0$, \eqref{eq1.1} reduces to the classical MV-SDE. MV-SDEswCN widely serve as fundamental models in mean-field games, mean-field control, finance, quantum mechanics and quantum chemistry (see, e.g., \cite{BLY2022,BLM2023,CG2019,DKL2014,D1975,DLW2024,HSS2021a,HSS2021,M1966}).  
Compared with MV-SDEs, the stochastic nature of conditional distributions brings new challenges to the ergodic analysis of the MV-SDEswCN \eqref{eq1.1}. This paper aims to establish the existence of the unique invariant measure and the strong law of large numbers for the MV-SDEwCN \eqref{eq1.1}, alongside the uniform-in-time propagation of chaos.

The existence of the unique invariant measure is a crucial dynamical property of SDEs that lays the foundation for ergodicity. Intuitively, the ergodicity reveals that the average of the system  in time tends to the spatial average (see \cite{DZ1996}), a property that is highly beneficial to swap time and space.
This characteristic has a wide range of applications in statistical inference and associated long-term behaviour connected to metastability (see, e.g., \cite{ABG2012,BPP2023,LMW2021,T2013}). 
Thus, the ergodicity of the MV-SDEswCN \eqref{eq1.1} have attracted much attention. However, owing to the stochastic nature of conditional distributions, the solution $X_t$ to the MV-SDEwCN \eqref{eq1.1} lacks the Markov property. Consequently, existing literature has primarily focused on the invariant measure of the conditional distribution $\mathcal{L}^1(X_t)$ under restrictive coefficient assumptions. For instance, assuming constant diffusion coefficients  $g$ and $g^0$ alongside a convolution-type drift, \cite{M} proved the existence and uniqueness of such an invariant measure. Similar results were obtained by \cite{DTM} for conditional expectation-type drifts with constant diffusions. More recently, \cite{BW} established the exponential contractivity of the conditional distribution when 
$g^0$ is constant and $g$ depends solely on the state. For extensions to jump-diffusion settings, we refer to \cite{BLW}.  Notably, Wang \cite{W2021} demonstrated the Markov property of the solution couple $(X_t, \mathcal{L}^1(X_t))$ for the MV-SDEwCN \eqref{eq1.1} with $g\equiv0$, and subsequently proved its exponential ergodicity. Motivated by this research direction, our primary objective is to establish the existence and uniqueness of the invariant measure and the exponential ergodicity for the solution couple $(X_t, \mathcal{L}^1(X_t))$ for the more general MV-SDEwCN \eqref{eq1.1}, where both the drift and diffusion coefficients depend on the state and the measure. 


Apart from investigating the exponential ergodicity, we are also interested in the probabilistic
limit behaviors, such as the strong law of large numbers(LLN) for the MV-SDEwCN \eqref{eq1.1}. As a fundamental topic in the study of Markov processes, the strong LLN primarily concerns the convergence of time averages as time approaches infinity, and has been extensively studied in the literature for time-homogeneous Markov processes \cite{CDHS2025,KW2012,K2018,S2006}. Specifically, for time-homogeneous Markov processes $\{x_t\}_{t\geq0}$ that exhibit exponential ergodicity with respect to their invariant measure $u^*$, it is well-established that the time average 
$\frac{1}{T}\int_0^T\psi(x_t)\mathrm{d}t$ converges almost surely to the ergodic limit  $\int \psi\mathrm{d}u^*$, i.e.
$$\lim_{T \to \infty} \frac{1}{T} \int_0^T \psi(x_t) \mathrm{d}t = \int  \psi\mathrm{d}u^*\quad \text{a.s.}$$ 
However, because the solution $X_t$ to the MV-SDEwCN \eqref{eq1.1} is not a Markov process, existing theory does not apply directly. To the best of our knowledge, few results exist concerning the strong LLN for this system. Consequently, our second objective is to establish this result.

The conditional propagation of chaos serves as the fundamental bridge between the MV-SDE \eqref{eq1.1} and interacting particle systems (IPS) with common noise (see \cite{CD2018b,CD2022a,CD2022b,S1991}). Precisely, consider the associated IPSwCN:
\begin{align}\label{eq5.1}
\mathrm{d}X^{i,N}_t=f(X^{i,N}_t,u_t^N)\mathrm{d}t+g(X^{i,N}_t,u_t^N)\mathrm{d}B^i_t+g^0(X^{i,N}_t,u_t^N)\mathrm{d}B^{0}_t, \quad i=1,\cdots,N,
\end{align}
where $u_t^N=\frac{1}{N}\sum_{j=1}^N\delta_{X^{j,N}_t}$ represents the empirical measure of the particles, and $(B^i_t)_{i\geq 1}$ is a sequence of independent Brownian motions independent of $B^0_t$. The conditional propagation of chaos asserts that as $N\to\infty$, the particles in the IPSwCN become conditionally independent, and their empirical conditional distributions converge to the conditional law of the solution to the limiting MV-SDE. Most existing literature focuses on the finite-time conditional propagation of chaos (see, e.g., \cite{BSW2023,CD2018b,CF2016,ELL2021,H2025,KNRS2022,VHP2022}). Nevertheless, the uniform-in-time conditional propagation of chaos has received comparatively less attention. To the best of our knowledge, existing uniform-in-time results are restricted to specific settings: \cite{M} established this theory for \eqref{eq1.1} with constant diffusions and a convolution-type drift, while \cite{BW} recently extended it to the case of a constant $g^0$ and a state-dependent $g$. The final objective is to establish the uniform-in-time propagation of chaos for a significantly more general class of MV-SDEs \eqref{eq1.1}. 

The dynamics of the MV-SDEswCN  \eqref{eq1.1} depend intrinsically on both the state trajectory and its conditional distribution, introducing more randomness. This intricate coupling implies that the solution process $X_t$ of \eqref{eq1.1} fails to possess the time-homogeneous Markov property. This fundamental distinction renders standard probabilistic tools---such as localization procedures and the Yamada-Watanabe theorem---inapplicable, thereby posing a significant challenge to the dynamical analysis. To circumvent this issue, we lift the system to the solution couple $(X_t,\mathcal{L}^1(X_t))$. Inspired by Wang's approach of constructing semigroups via distribution flows \cite{W2018}, we build a time-homogeneous semigroup by the measure couple $(\mathcal{L}(X_t),\mathcal{L}(\mathcal{L}^1(X_t)))$ for this lifted process, which paves the way for proving the existence and uniqueness of the invariant measure. A crucial prerequisite for this construction is the uniqueness of the measure pair $(\mathcal{L}(X_t),\mathcal{L}(\mathcal{L}^1(X_t)))$. Although weak uniqueness has been investigated \cite{CD2018b,HSS2021}, the case with nonlinear coefficients depending on both the state and the measure remains unexplored. Thus, we first establish the weak uniqueness of the solution couple. Building upon this, we construct an operator $\Phi_t$ by the joint distribution flow and verify that $\{\Phi_t\}_{t\geq 0}$ forms a Markov semigroup. Ultimately, this semigroup framework not only enables us to establish the existence and uniqueness of the invariant measure, but also allows us to prove the exponential ergodicity.  Notably, despite having established the invariant measure theory, the coupling between the state and the measure in the coefficients still necessitates delicate handling during the proof of the strong LLN. To circumvent this difficulty, we introduce a decoupling technique by freezing the conditional distribution process. Through this approach, we successfully establish the strong LLN for the solution couple 
$(X_t,\mathcal{L}^1(X_t))$ and the solution $X_t$ itself. Crucially, this provides a much more comprehensive characterization of the long-time behavior than analyzing the invariant measure of the conditional distribution alone.

The main contributions of this paper are summarized as follows:
\begin{itemize}
	\item We investigate a large class of MV-SDEswCN \eqref{eq1.1}, where both the drift and diffusion coefficients depend on the state and the measure, and the drift term exhibits at most polynomial growth.
	
	\item We propose a novel methodological framework for establishing the ergodicity property of the MV-SDEswCN \eqref{eq1.1}. By constructing a product measure space and defining an appropriate operator semigroup, we simultaneously capture the invariant measure information for both the conditional distribution $\mathcal{L}^1(X_t)$ and the state $X_t$. Building upon this framework, we further establish the strong LLN through a decoupling approach

	\item We establish the uniform-in-time conditional propagation of chaos for this general class of equations. Furthermore, we prove the convergence of the single-particle distribution and the empirical measure of the IPSwCN \eqref{eq5.1} to the marginals of the invariant measure of the limiting MV-SDEswCN  \eqref{eq1.1}, thereby validating the particle approximation on infinite time scales.
\end{itemize}

The remainder of this paper is organized as follows. Section \ref{Sect.1} introduces some notations and definitions.  Section \ref{Sect.2} establishes the existence and uniqueness of the invariant measure for \eqref{eq1.1}. Section \ref{Sect.3} investigates the strong law of large numbers of the MV-SDEwCN \eqref{eq1.1}.  Section \ref{Sect.4} is dedicated to proving the long-time convergence of the IPSwCN \eqref{eq5.1} to the MV-SDEwCN \eqref{eq1.1}. Section \ref{Sect.5} concludes this paper with two examples.

\section{Preliminaries}\label{Sect.1}
We first give the notations used in this paper. For any $a,b\in\mathbb{R}$, let $a\vee b=\max\{a,b\}$. Let $K$ denote the positive constant which may take different values in different places. We use $K_r$ to emphasize the dependence on parameter $r$. Let $\delta_x$ denote the Dirac measure at $x$. The Euclidean norm of a vector and the Hilbert-Schmidt norm
of a matrix are both denoted by $|\cdot|$. 
For the inner product of two vectors $x,y \in \mathbb{R}^d$, we write $x^Ty$. Let $\mathcal{C}(\mathbb{R}^d;\mathbb{R})$ be the family of the continuous functions from $\mathbb{R}^d$ to $\mathbb{R}$ and   $\mathcal{C}^{2}(\mathbb{R}^d;\mathbb{R})$ be the family of continuously twice differentiable functions from $\mathbb{R}^d$ to $\mathbb{R}$. For a function $V\in \mathcal{C}^{2}(\mathbb{R}^d;\mathbb{R})$,  we set $$V_x =\left(\frac{\partial V}{\partial x_1},\cdots,\frac{\partial V}{\partial x_d}\right), ~~\quad V_{xx}=\left(\frac{\partial^2 V}{\partial x_i x_j}\right)_{d\times d}.$$
Let $\mathcal{C}_c^2(\mathbb{R}^d;\mathbb{R})$ be the family of  continuously twice differentiable functions with compact support
from $\mathbb{R}^d$ to $\mathbb{R}$. Let $\mathcal{C}_c^\infty(\mathbb{R}^d;\mathbb{R})$ be the family of smooth functions with compact support
from $\mathbb{R}^d$ to $\mathbb{R}$. 
Let $(E, \rho)$ be a metric space and
let $\mathcal{P}(E)$ denote the space of all probability measures taking values on $(E, \mathcal{B}(E))$ equipped with the weak topology. For $k\geq1$, let
$$
\mathcal{P}_k (E):=\left\{{u} \in \mathcal{P}(E): {\int_{E} \rho^k(x, x_0) u(\mathrm{d}x) }<\infty\right\},
$$
for some $x_0\in E$.
Especially, for $ u, v \in \mathcal{P}_k(\mathbb{R}^d)$, define
$${W}_k(u, v)\!=\!\inf_{\pi \in \Pi(u, v)}\Big(\int_{\mathbb{R}^d \times \mathbb{R}^d}|x-y|^k \pi(\mathrm{d} x, \mathrm{d} y)\Big)^{1/k},
$$
where $\Pi(u, v)$ is the set of all couplings of $u$ and $v$. By Theorem 6.18 in \cite{V2009}, $(\mathcal{P}_k(\mathbb{R}^d),{W}_k)$ is a Polish space.
Furthermore, For $\mu, \upsilon\in\mathcal{P}_k(\mathcal{P}_k(\mathbb{R}^d))$, define the Wasserstein distance 
$$
\mathbb{W}_k(\mu, \upsilon)\!=\!\inf_{\pi \in \Pi(\mu, \upsilon)}\Big(\int_{\mathcal{P}_k(\mathbb{R}^d)\times \mathcal{P}_k(\mathbb{R}^d)}{W}_k^k(u,v) \pi(\mathrm{d} u, \mathrm{d} v)\Big)^{1/k}, 
$$
where $\Pi(\mu, \upsilon)$ is the set of all couplings of $\mu$ and $\upsilon$. Theorem 6.18 in \cite{V2009} pointed out that $(\mathcal{P}_k(\mathcal{P}_k(\mathbb{R}^d)),\mathbb{W}_k)$ is also a Polish space.

We now give the probability framework of the paper.
Let $(\Omega^0, \mathcal{F}^0, \mathbb{P}^0)$ and $(\Omega^1, \mathcal{F}^1,\mathbb{P}^1)$ be two complete probability spaces endowed with complete and right-continuous filtration $(\mathcal{F}_t^0)_{t \geq 0}$ and $(\mathcal{F}_t^1)_{t \geq 0}$, respectively. Here $(B^0_t)_{t\geq 0}$ and $(B_t)_{t\geq 0}$  are two independent $d$-dimensional Brownian motions defined on  $\left(\Omega^0, \mathcal{F}^0, \mathbb{P}^0\right)$ and $\left(\Omega^1, \mathcal{F}^1, \mathbb{P}^1\right)$, respectively.  
Define the product space $(\Omega, \mathcal{F}, \mathbb{P})$, where
$\Omega=\Omega^0 \times \Omega^1,$ {$(\mathcal{F},\mathbb{P})$} is the completion of $(\mathcal{F}^0\otimes \mathcal{F}^1,\mathbb{P}^0 \otimes \mathbb{P}^1)$, and $(\mathcal{F}_t)_{t\geq 0}$ is the complete and right-continuous augmentation of $(\mathcal{F}^0_t\otimes\mathcal{F}^1_t)_{t\geq 0}$. 
We denote by $\mathbb{E}$ and $\mathbb{E}^0$ the expectations with respect to $\mathbb{P}$ and $\mathbb{P}^0$, respectively.
For a random variable $\xi$, we denote its probability distribution by $\mathcal{L}(\xi)$. For $k\geq1$, let $L^k(\Omega;\mathbb{R}^d)$ be the family of random variables $\xi:\Omega\rightarrow \mathbb{R}^d$ satisfying $\mathbb{E}|\xi|^k<\infty$, and ${L^k(\Omega^0;\mathcal{P}_k(\mathbb{R}^d))}$ be the family of random variables $u:\Omega^0\rightarrow \mathcal{P}_k(\mathbb{R}^d)$ satisfying  $\mathbb{E}^0({W}_k^k(u,\delta_{\bf 0}))<\infty$, where $\delta_{\bf 0}$ denote the Dirac measure at ${\bf 0}\in\mathbb{R}^d$.
For random variable $\xi:\Omega\rightarrow \mathbb{R}^d$, 
we now define the mapping $\mathcal{L}^1(\xi)(\omega^0)= \mathcal{L}(\xi(\omega^0, \cdot))$ for any $\omega^0\in\Omega^0$.
By Lemma 2.4 in \cite{CD2018b}, the mapping $\mathcal{L}^1(\xi)(\cdot)$ is not only a random variable from $\Omega^0$ to $\mathcal{P}(\mathbb{R}^d)$ but also the conditional distribution of $\xi$ given $\mathcal{F}^0$.
With a slight abuse of notation, we do not distinguish a random variable $\xi$ on $\Omega^0$ from its natural extension on $\Omega$, and a sub-$\sigma$-algebra $\mathcal{G}^0$ of $\mathcal{F}^0$ from $\mathcal{G}^0\otimes \{\varnothing,\Omega^1\}\subseteq\mathcal{F}$.  Since the conditional distribution is involved, we note that for any $k>1$ and random variables $\xi_1,\xi_2\in L^k(\Omega;\mathbb{R}^d)$
\begin{align}\label{ls1}
	\begin{aligned}
		\mathbb{W}_k^k(\mathcal{L}(\mathcal{L}^1(\xi_1)),\mathcal{L}(\mathcal{L}^1(\xi_2)))&\leq\mathbb{E}^0W_k^k(\mathcal{L}^1(\xi_1),\mathcal{L}^1(\xi_2))
		\\&
		\leq\mathbb{E}^0\mathbb{E}^1|\xi_1-\xi_2|^k=\mathbb{E}|\xi_1-\xi_2|^k.
	\end{aligned}
\end{align}
We further assume the same probability frame on a new space $\tilde{\Omega}=\tilde{\Omega}^0\times \tilde{\Omega}^1$ as on ${\Omega}$. 
Next, we give the definition of the solution to the MV-SDEwCN \eqref{eq1.1}.
\begin{De}
(1) An $\mathbb{R}^d$-valued stochastic process $(X_t^s)_{t\geq s} $  is called a  strong solution to the MV-SDEwCN \eqref{eq1.1} with initial data $
		\varsigma\in L^2(\Omega;\mathbb{R}^d)$ starting at $s$, if 
		it is continuous $(\mathcal{F}_t)_{t\geq s} $-adapted  stochastic process satisfying that for any $T\geq  s$, $\mathbb{E}[\sup_{s\leq t\leq T}|X_t^s|^2]<\infty$, 
		$$
		\mathbb{E}\int_s^T\big(|f(X_t^s,\mathcal{L}^1(X_t^s))|+|g(X_t^s, \mathcal{L}^1(X_t^s))|^2
		+|g^0( X_t^s,\mathcal{L}^1(X_t^s))|^2\big) \mathrm{d} t<\infty,
		$$
		and  in $\mathbb{P}$-a.s.
		\begin{align*}
			X_t^s&=\varsigma+\int_s^tf(X^s_r ,\mathcal{L}^1(X^s_r))
			\mathrm{d} r+\int_s^tg(X^s_r ,\mathcal{L}^1(X^s_r))\mathrm{d} B_r
			+\int_s^t g^0(X^s_r,\mathcal{L}^1(X^s_r)) \mathrm{d} B_r^0.
		\end{align*}
		(2) A quadri-tuple $(\tilde{X}^s_t,\mathcal{L}^1(\tilde{X}^s_t),\tilde{B}_t^0,\tilde{B}_t)_{t\geq s}$ with respect to the complete probability space $(\tilde{\Omega},\tilde{\mathcal{F}},\tilde{\mathbb{P}})$ is called a weak solution to the MV-SDEwCN \eqref{eq1.1}  with initial data $
		\tilde{\varsigma}\in L^2(\tilde{\Omega};\mathbb{R}^d)$ starting at $s$, where $\tilde{B}_t^0,\tilde{B}_t$ are two dependent $d$-dimensional Brownian motion on the complete probability spaces  $(\tilde{\Omega}^0,\tilde{\mathcal{F}}^0,\tilde{\mathbb{P}}^0)$ and  $(\tilde{\Omega}^1,\tilde{\mathcal{F}}^1,\tilde{\mathbb{P}}^1)$, if
		for any $T\geq  s$, $\mathbb{\tilde{E}}[\sup_{s\leq t\leq T}|\tilde{X}^s_t|^2]<\infty$, 
		$$
		\mathbb{\tilde{E}}\int_s^T\big(|f(\tilde{X}^s_t,\mathcal{L}^1(\tilde{X}^s_t))|+|g( \tilde{X}^s_t, \mathcal{L}^1(\tilde{X}^s_t))|^2
		+|g^0( \tilde{X}^s_t ,\mathcal{L}^1(\tilde{X}^s_t))|^2\big) \mathrm{d} t<\infty,
		$$
		and 
		$\mathbb{\tilde{P}}$-a.s.
		\begin{align*}
			\tilde{X}_t^s&=\tilde{\varsigma}+\int_s^t f(\tilde{X}^s_r ,\mathcal{L}^1(\tilde{X}^s_r))
			\mathrm{d} r+\int_s^t g(\tilde{X}^s_r ,\mathcal{L}^1(\tilde{X}^s_r))\mathrm{d} \tilde{B}_r
			+\int_s^t g^0(\tilde{X}^s_r,\mathcal{L}^1(\tilde{X}^s_r)) \mathrm{d} \tilde{B}_r^0.
		\end{align*}
		(3) The MV-SDEwCN \eqref{eq1.1} is called weakly unique in $\mathcal{P}_k(\mathbb{R}^d)\times\mathcal{P}_k(\mathcal{P}_k(\mathbb{R}^d))$, if for any pair of weak solutions $( {X}^s_t,\mathcal{L}^1( {X}^s_t), {B}_t, {B}_t^0)_{t\geq s}$
		and $(\tilde{X}^s_t,\mathcal{L}^1(\tilde{X}^s_t),\tilde{B}_t,\tilde{B}_t^0)_{t\geq s}$ of \eqref{eq1.1},  \\
		$$(\mathcal{L}( {X}^s_s),\mathcal{L}(\mathcal{L}^1( {X}^s_s)))= (\mathcal{L}(\tilde{X}^s_s),\mathcal{L}(\mathcal{L}^1(\tilde{X}^s_s))) \in \mathcal{P}_k(\mathbb{R}^d)\times\mathcal{P}_k(\mathcal{P}_k(\mathbb{R}^d))$$ implies that 
		$$ (\mathcal{L}( {X}^s_t),\mathcal{L}(\mathcal{L}^1( {X}^s_t)))=(\mathcal{L}(\tilde{X}^s_t),\mathcal{L}(\mathcal{L}^1(\tilde{X}^s_t)))\in \mathcal{P}_k(\mathbb{R}^d)\times\mathcal{P}_k(\mathcal{P}_k(\mathbb{R}^d)), \quad\forall t > s.$$
\end{De}

Consider the stochastic partial differential equation (SPDE) corresponding to the MV-SDEwCN \eqref{eq1.1} 
\begin{align}\label{r3}
	\mathrm{d}\langle u_t,\varphi\rangle=\langle u_t,L_{u_t} \varphi\rangle\mathrm{d}t+\langle u_t,\varphi_x g^0(\cdot, u_t)\rangle\mathrm{d}B_t^0,\quad t\geq0,~\varphi\in \mathcal{C}_c^\infty(\mathbb{R}^d;\mathbb{R}),
\end{align}
where for any $u\in \mathcal{P}(\mathbb{R}^d)$, $\langle u ,\varphi\rangle=\int_{\mathbb{R}^d}\varphi(x) u (\mathrm{d}x)$ and 
\begin{align*}
	L_{ u } \varphi:=\varphi_x f(\cdot, u )+\frac{1}{2}\operatorname{trace}[(g(\cdot, u )(g(\cdot, u ))^T+g^0(\cdot, u )(g^0(\cdot, u ))^T)\varphi_{xx}^T].
\end{align*}
Next we give some definitions on the solutions of the SPDE \eqref{r3} in different means, including probability strong solution, probability weak solution and probability weak uniqueness.
\begin{De}\label{d1}
		(1) For any $s\geq0$, a continuous $\mathcal{P}(\mathbb{R}^d)$-valued process $( u_t^s)_{t\geq s}$ defined on
		$(\Omega^0,  \mathcal{F}^0, \mathbb{P}^0)$ is called a probability strong solution to the SPDE \eqref{r3} starting at $s$, if
		\begin{align*}
			\mathbb{E}^0\bigg[\int_{s}^t\int_{\mathbb{R}^d}(|f(x,u^s_r)|+|g(x, u^s_r)|^2+|g^0(x, u^s_r)|^2) u^s_r(\mathrm{d}x)\mathrm{d}r\bigg]<\infty, \quad t> s,
		\end{align*}
		and~$\mathbb{P}^0$-a.s., $\forall t\geq s,~\varphi\in \mathcal{C}_c^\infty(\mathbb{R}^d;\mathbb{R})$
		\begin{align*}
			&\langle u^s_t,\varphi\rangle=\langle u^s_s,\varphi\rangle+\int_s^t\langle u^s_r,L_{ u^s_r}
			\varphi\rangle\mathrm{d}r
			+\int_s^t\langle u^s_r,\varphi_x g^0(\cdot, u^s_r)\rangle\mathrm{d}B_r^0.
		\end{align*}
		We say the SPDE \eqref{r3} has probability strong existence and uniqueness in $L^k(\Omega^0;$ $ \mathcal{P}(\mathbb{R}^d))$, if for any 
		initial data $u_{s}^{s}\in L^k(\Omega^0; \mathcal{P}(\mathbb{R}^d)),$ 
		the SPDE \eqref{r3}   has a unique solution 
		$( u^s_t)_{t\geq s}$ taking values in $ L^k(\Omega^0;\mathcal{P}(\mathbb{R}^d))$ for any $t\geq s$. \\
		(2) A couple $(\tilde{u}^s_t,\tilde{B}_t^0)_{t\geq s}$ is called a probability weak solution to the SPDE \eqref{r3} starting at $s$, where $\tilde{B}_t^0$ is a $d$-dimensional Brownian motion on the complete probability space $(\tilde{\Omega}^0,\tilde{\mathcal{F}}^0,\tilde{\mathbb{P}}^0)$ if it satisfies  
		\begin{align*}
			\mathbb{\tilde{E}}^0\bigg[\int_{s}^t\int_{\mathbb{R}^d}(|f(x,\tilde{u}^s_r)|+|g(x, \tilde{u}^s_r)|^2+|g^0(x, \tilde{u}^s_r)|^2) \tilde{u}^s_r(\mathrm{d}x)\mathrm{d}r\bigg]<\infty, \quad t> s,
		\end{align*}
		and~$\mathbb{\tilde{P}}^0$-a.s. 
		\begin{align*}
			&\langle\tilde{u}^s_t,\varphi\rangle\!=\!\langle\tilde{u}^s_s,\varphi\rangle\!+\!\int_s^t\!\!\langle\tilde{u}^s_r,L_{\tilde{u}^s_r} \varphi\rangle\mathrm{d}r
			\!+\!\int_s^t\!\!\langle\tilde{u}^s_r, \varphi_x g^0(\cdot,\tilde{u}^s_r)\rangle\mathrm{d}\tilde{B}_r^0,
			~\forall t\geq s,~\varphi\in \mathcal{C}_c^\infty(\mathbb{R}^d;\mathbb{R}).
		\end{align*}
	 (3) SPDE \eqref{r3} starting at $s$ is said to have probability weak  uniqueness  in $\mathcal{P}_k(\mathcal{P}_k(\mathbb{R}^d)\!)$, if for any pair of weak solutions 
		$({u}^{s}_{t}, {B}^0_t)_{t \geq s}$ and
		$(\tilde{u}^{s}_{t}, \tilde{B}^0_t)_{t \geq s}$, 
		$$\mathcal{L}({{u}^s_{ s}})=\mathcal{L}({\tilde{u}^{s}_{s}})\in \mathcal{P}_k(\mathcal{P}_k(\mathbb{R}^d))$$ implies 
		$$\mathcal{L}({{u}^s_{t}})=\mathcal{L}({\tilde{u}^s_{t}})\in \mathcal{P}_k(\mathcal{P}_k(\mathbb{R}^d))$$
		for any $t > s$.
\end{De}

We end this section by stating the necessary assumptions on the coefficients and presenting some useful results for the MV-SDEwCN \eqref{eq1.1}.
\begin{itemize}
	\item[\bf(A1)]
	There exist positive constants $l\geq2$, $c_1,c_2$ such that for any $(x,u)\in\mathbb{R}^d\times\mathcal{P}_{2}(\mathbb{R}^d)$,
	\begin{align*}
		|f(x, u)|^2&\leq c_1(1+|x|^l+{W}_2^2( u,\delta_{\bf 0})),\\
		|g(x, u)|^2+|g^0(x, u)|^2&\leq c_2(1+|x|^2+{W}_2^2( u,\delta_{\bf 0})).
	\end{align*}
	\item [\bf (A2)]
	$f$ is locally Lipschitz in $\mathbb{R}^d\times \mathcal{P}_2(\mathbb{R}^d)$. There exist positive constants $p> 2$, $c_3>c_4$ such that for any $(x, u),(y,v)\in\mathbb{R}^d\times\mathcal{P}_{2}(\mathbb{R}^d)$,
	\begin{align*}
		&2(x-y)^T(f(x, u)-f(y, v))+(p-1)|g(x, u)-g(y, v)|^2\\&+(p-1)|g^0(x, u)-g^0(y, v)|^2)
		\leq -c_3|x-y|^2+c_4{W}_2^2( u, v).
	\end{align*}
\end{itemize}

\begin{rem}
	Under {\bf (A2)}, one notices that both $g$ and $g^0$ are locally Lipschitz in $\mathbb{R}^d\times \mathcal{P}_2(\mathbb{R}^d)$. Moreover,  there exists positive constants $c_5$ and $c_6>c_7$ such that for any $(x,u)\in\mathbb{R}^d\times\mathcal{P}_{2}(\mathbb{R}^d)$,
	\begin{align}\label{A2}
		&2x^Tf(x, u)+(p-1)(| g(x, u)|^2+|g^0(x, u)|^2)
		\leq c_5-c_6|x|^2+c_7{W}_2^2( u,\delta_{\bf 0}).
	\end{align}
\end{rem}

By virtue of Theorem 2.1 in \cite{KNRS2022}, we obtain the existence and uniqueness of the solution to the MV-SDEwCN \eqref{eq1.1}.
	\begin{lem}\label{le2.4}
		Under {\bf(A1)} and {\bf(A2)}, the MV-SDEwCN \eqref{eq1.1} with initial value $\varsigma \in L^p(\Omega;\mathbb{R}^d)$ has a unique strong solution $(X_t)_{t\geq0}$ satisfying 
		$\mathbb{E}|X_t|^p<\infty$  for any $t\geq 0$.
	\end{lem}
	Next we consider the propagation of chaos of the MV-SDEwCN \eqref{eq1.1}.  We first introduce the corresponding non-IPSwCN
	\begin{align}\label{eq5.2}
		\begin{aligned}
			\mathrm{d}{X}^{i}_t&= f({X}^{i}_t,\mathcal{L}^1({X}^{i}_t))\mathrm{d} t
			+g({X}^{i}_t,\mathcal{L}^1({X}^{i}_t))\mathrm{d}B^i_t
			+g^0({X}^{i}_t,\mathcal{L}^1({X}^{i}_t))\mathrm{d}B^{0}_t,\quad i\geq 1,
		\end{aligned}
	\end{align}
	with the initial condition  $X_0^i=\varsigma^i$, where  $(\varsigma^i, B^i_t)_{i\geq1}$ is a sequence of independent copies of $(\varsigma, B_t)$ with $\varsigma\in L^p(\Omega;\mathbb{R}^d)$.  Note that the IPSwCN \eqref{eq5.1} can be directly reformulated as a classical SDE in $\mathbb{R}^{d\times N}$,  whose the drift and diffusion coefficients (valued in $\mathbb{R}^{d\times N}$ and $\mathbb{R}^{d\times d\times N}$, respectively) inheriting the conditions {\bf (A1)} and {\bf (A2)} (see, e.g. \cite{KNRS2022}). Consequently, IPSwCN \eqref{eq5.1} with initial data $X_0^{i, N}=\varsigma^i~ (i=1, \cdots, N)$ has a  unique solution $\mathbf{X}^N_t=(X^{1,N}_t,\cdots,X^{N,N}_t)$ for $t\geq 0$ by \cite{M2008}. Then the propagation of chaos can be easily derived from Proposition 1 in \cite{KNRS2022}.
	\begin{lem} Let {\bf(A1)} and {\bf(A2)} hold. Then for any $2\leq q<p$,
		\begin{align*}
			&\lim_{N\rightarrow \infty}\sup_{t\in[0,T]}\mathbb{E}|{X}^i_t -X^{i,N}_t|^{q}=0, ~~~~~~\lim_{N\rightarrow \infty}\sup_{t\in[0,T]}\mathbb{E}{W}_{q}^{q}\big(\mathcal{L}^1({X}^{i}_t), {u_t^N} \big)=0,
		\end{align*}
		where $ (X^{1,N}_t,\cdots,X^{N,N}_t) $ is the solution  to the IPS \eqref{eq5.1} and $({X}^i_t)_{i=1,\cdots,N}$ are the solutions to the non-IPSwCN \eqref{eq5.2}, and  $u_t^N=\frac{1}{N}\sum_{j=1}^N\delta_{X^{j,N}_t}$.
	\end{lem}

\section{The invariant measure}\label{Sect.2}
This section establishes the existence and uniqueness of the invariant measure for the MV-SDEwCN \eqref{eq1.1}. As the state process $X_t$ lacks the Markov semigroup property, the standard approach to invariant measures is not directly applicable. We therefore lift the system to the solution couple $(X_t, \mathcal{L}^1(X_t))$.  We first prove that the uniqueness of the corresponding measure couple $(\mathcal{L}(X_t), \mathcal{L}(\mathcal{L}^1(X_t)))$ (the weak uniqueness defined in Preliminaries). By leveraging the weak uniqueness of this couple, we define an operator via its distribution flow and verify that it forms a time-homogeneous Markov semigroup. This semigroup framework then allows us to establish the existence and uniqueness of the invariant measure for the solution couple.
	
	We begin this section with showing that the solution couple of MVSDEwCN \eqref{eq1.1} is weakly unique.
	\begin{lem}\label{ls}
		Let {\bf(A1)} and {\bf(A2)} hold with $p> l/2$.  Then the solution couple of MVSDEwCN \eqref{eq1.1} has weak uniqueness. 
	\end{lem}
	
	\begin{proof}Since the proof is rather technical, we divide it into three steps.\\
		{\bf Step 1.} Any conditional distribution of the solution to the MV-SDEwCN \eqref{eq1.1} with respect to common noise is the solution to SPDE \eqref{r3}. Let Brownian motions $B_t^i,B_t^0$ be defined on probability space $(\Omega,\mathcal{F},\mathbb{P})$ and $\tilde{B}_t^i,\tilde{B}_t^0$ be defined on $(\tilde{\Omega},\tilde{\mathcal{F}},\tilde{\mathbb{P}})$, respectively. Let $({X}_t,\mathcal{L}^1({X}_t))_{t\geq 0}$ be the  solution couple to the MV-SDEwCN 
		\begin{align}\label{eqX}
			\mathrm{d}X_t&=f(X_t ,\mathcal{L}^1(X_t))
			\mathrm{d} t+g(X_t ,\mathcal{L}^1(X_t))\mathrm{d} B_t
			+g^0(X_t,\mathcal{L}^1(X_t)) \mathrm{d} B_t^0,
		\end{align}
		with initial value $\varsigma\in L^p(\Omega;\mathbb{R}^d)$ and $({\tilde{X}}_t,\mathcal{L}^1({\tilde{X}}_t))_{t\geq 0}$ be the  solution couple to MV-SDEwCN 
		\begin{align}\label{eqX1}
			\mathrm{d}\tilde{X}_t&=f(\tilde{X}_t ,\mathcal{L}^1(\tilde{X}_t))
			\mathrm{d} t+g(\tilde{X}_t ,\mathcal{L}^1(\tilde{X}_t))\mathrm{d} \tilde{B}_t
			+g^0(\tilde{X}_t,\mathcal{L}^1(\tilde{X}_t)) \mathrm{d} \tilde{B}_t^0,
		\end{align} 
		with initial value $\tilde{\varsigma}\in L^p(\tilde{\Omega};\mathbb{R}^d)$, where $\mathcal{L}(\varsigma)=\mathcal{L}(\tilde{\varsigma})$ and $\mathcal{L}(\mathcal{L}^1(\varsigma))=\mathcal{L}(\mathcal{L}^1(\tilde{\varsigma}))$. 
		Applying {\bf(A1)}, the H\"older inequality and Lemma \ref{le2.4} yields that for any $T>0$
		\begin{align}\label{ls2}
			\begin{aligned}
				&\int_0^T\mathbb{E}(|f(X_t, \mathcal{L}^1(X_t))|+|g(X_t, \mathcal{L}^1(X_t))|^{2}+|g^0(X_t, \mathcal{L}^1(X_t))|^{2})\mathrm{d}t
				\\&\leq \int_0^T\mathbb{E}\Big(c_1\Big(1\!+\!|X_t|^l\!+\!{W}_2^2(\mathcal{L}^1(X_t),\delta_{\bf 0})\Big)^{\frac{1}{2}}
				\!+\!
				2\Big(c_2\Big(1\!+\!|X_t|^2\!+\!{W}_2^2(\mathcal{L}^1(X_t),\delta_{\bf 0})\Big)\Big)\Big)\mathrm{d}t		
				\\&\leq K \int_0^T\Big( 1\!+\! \mathbb{E}|X_t|^\frac{l}{2}\!+\!\mathbb{E}|X_t|^{2}\!+\!
				\mathbb{E}{W}_2^2(\mathcal{L}^1(X_t),\delta_{\bf 0}))
				\Big)\mathrm{d}t
				\\&\leq K\int_0^T\big(1+ \mathbb{E}|X_t|^p\big)\mathrm{d}t<\infty.
			\end{aligned}
		\end{align}
		Similarly, 
		\begin{align*}
			&\int_0^T\mathbb{E}(|f(\tilde{X}_t, \mathcal{L}^1(\tilde{X}_t))|+|g(\tilde{X}_t, \mathcal{L}^1(\tilde{X}_t))|^{2}+|g^0(\tilde{X}_t, \mathcal{L}^1(\tilde{X}_t))|^{2})\mathrm{d}t
			\nn\\&\leq K\int_0^T\big(1+ \mathbb{E}|\tilde{X}_t|^p\big)\mathrm{d}t<\infty.
		\end{align*}
		Then, by Proposition 1.2  in \cite{LSZ2023} and Proposition 2.11 in \cite{KNRS2022}, the conditional distribution 
		$(\mathcal{L}^1(X_t))_{t\geq 0}$ and $(\mathcal{L}^1(\tilde{X}_t))_{t\geq 0}$ of the solutions $X_t$ and $\tilde{X}_t$ to \eqref{eqX} and \eqref{eqX1} are the unique strong probability solutions to the SPDE 
		\begin{align}\label{equ}
			&\langle u^s_t,\varphi\rangle=\langle u^s_s,\varphi\rangle+\int_s^t\langle u^s_r,L_{ u^s_r}
			\varphi\rangle\mathrm{d}r
			+\int_s^t\langle u^s_r,\varphi_x g^0(\cdot, u^s_r)\rangle\mathrm{d}B_r^0,
		\end{align}
		and 	\begin{align}\label{equ1}
			&\langle\tilde{u}^s_t,\varphi\rangle\!=\!\langle\tilde{u}^s_s,\varphi\rangle\!+\!\int_s^t\!\!\langle\tilde{u}^s_r,L_{\tilde{u}^s_r} \varphi\rangle\mathrm{d}r
			\!+\!\int_s^t\!\!\langle\tilde{u}^s_r, \varphi_x g^0(\cdot,\tilde{u}^s_r)\rangle\mathrm{d}\tilde{B}_r^0,
		\end{align}
		where $t\geq 0,~\varphi\in \mathcal{C}_c^\infty(\mathbb{R}^d;\mathbb{R})$, 
		with initial value $\mathcal{L}^1(\varsigma)$ and $\mathcal{L}^1(\tilde{\varsigma})$, respectively.\\
		{\bf Step 2.} The SPDE \eqref{r3} is weakly unique. To this end, we only need to prove that  $\mathcal{L}(\mathcal{L}^1(X_t))=\mathcal{L}(\mathcal{L}^1(\tilde{X}_t)), $ for any $t\geq0$.
		Let $B_t^i, \tilde{B}_t^i,i=1,\cdots,N$ be the independent copies of $B_t, \tilde{B}_t$ in probability space $(\Omega,\mathcal{F},\mathbb{P})$ and $(\tilde{\Omega},\tilde{\mathcal{F}},\tilde{\mathbb{P}})$, respectively. Then the corresponding particle systems of \eqref{eqX} and \eqref{eqX1} are
		\begin{align*}
			\mathrm{d}X^{i,N}_t&=f(X^{i,N}_t,u_t^N)\mathrm{d}t+g(X^{i,N}_t,u_t^N)\mathrm{d}B^i_t
			+g^0(X^{i,N}_t,u_t^N)\mathrm{d}B^{0}_t,~~ i=1,\cdots,N,
		\end{align*}
		and 
		\begin{align*}
			\mathrm{d}\tilde{X}^{i,N}_t&=f(\tilde{X}^{i,N}_t,\tilde{u}_t^N)\mathrm{d}t+g(\tilde{X}^{i,N}_t,\tilde{u}_t^N)\mathrm{d}\tilde{B}^i_t
			+g^0(\tilde{X}^{i,N}_t,\tilde{u}_t^N)\mathrm{d}\tilde{B}^{0}_t,~~ i=1,\cdots,N,
		\end{align*}
		respectively, where $u_t^N=\frac{1}{N}\sum_{j=1}^N\delta_{X^{j,N}_t}$ and $\tilde{u}_t^N=\frac{1}{N}\sum_{j=1}^N\delta_{\tilde{X}^{j,N}_t}$.
		By the propagation of chaos, we have that for any $T\geq 0$
		\begin{align*}
			\lim_{N\rightarrow \infty}\sup_{t\in[0,T]}\mathbb{E}{W}_{2}^{2}\big(\mathcal{L}^1({X}^{i}_t), {u_t^N} \big)=0, \quad \lim_{N\rightarrow \infty}\sup_{t\in[0,T]}\mathbb{E}{W}_{2}^{2}\big(\mathcal{L}^1(\tilde{{X}}^{i}_t), {\tilde{u}_t^N} \big)=0.
		\end{align*}
		This implies that ${u_t^N},{\tilde{u}_t^N}$ converge to $\mathcal{L}^1({X}^{i}_t), \mathcal{L}^1(\tilde{{X}}^{i}_t)$ in distribution on $\mathcal{P}(\mathbb{R}^d)$, respectively. We claim that $\mathcal{L}(\mathcal{L}^1({X}^{i}_t))=\mathcal{L}(\mathcal{L}^1(\tilde{X}^{i}_t))$. In fact,
		it follows from the exchangeability of the particles and the existence and uniqueness of the solution of those two particle systems that
		$X^{i,N}_t$ and $\tilde{X}^{i,N}_t, i=1,\cdots,N,$ are identically distributed. This gives that  $u_t^N$ and $\tilde{u}_t^N$ are identically distributed for any $N\geq 0$. Then the convergence of ${u_t^N},{\tilde{u}_t^N}$ in distribution on $\mathcal{P}(\mathbb{R}^d)$, combined with the identical distribution of $u_t^N$ and $\tilde{u}_t^N$, shows that $\mathcal{L}^1({X}^{i}_t)$ and $\mathcal{L}^1(\tilde{{X}}^{i}_t)$ are identically distributed. \\
		{\bf Step 3.} The MV-SDEwCN is weakly unique. As proved in Step 2,  $\mathcal{L}(\mathcal{L}^1(X_t))=\mathcal{L}(\mathcal{L}^1(\tilde{X}_t))), \forall t\geq0$. This implies that for any bounded measurable function $\varphi$ from $\mathbb{R}^d$ to $\mathbb{R}$, 
		\begin{align*}
			\langle\mathcal{L}(X_t), \varphi\rangle&=\mathbb{E}\varphi(X_t)=\mathbb{E}^0\mathbb{E}^1(\varphi(X_t))=\mathbb{E}^0(\langle\mathcal{L}^1(X_t)(\omega^0), \varphi\rangle)\nn\\&=\int_{\mathcal{P}(\mathbb{R}^d)} \langle u, {\varphi}\rangle\mathcal{L}(\mathcal{L}^1(X_t))(\mathrm{d} u)=\int_{\mathcal{P}(\mathbb{R}^d)} \langle u, {\varphi}\rangle\mathcal{L}(\mathcal{L}^1(\tilde{X}_t))(\mathrm{d} u)
			\nn\\&=\mathbb{E}^0\mathbb{E}^1(\varphi(\tilde{X}_t))=\mathbb{E}\varphi(\tilde{X}_t)=\langle\mathcal{L}(\tilde{X}_t),\varphi\rangle,
		\end{align*}
		namely, $\mathcal{L}(X_t)=\mathcal{L}(\tilde{X}_t)$ for any $t\geq0$. This completes the proof.
		\end{proof}
		
	 We next define an operator via the distribution flow $(\mathcal{L}( {X}_t),\mathcal{L}(\mathcal{L}^1( {X}_t)))$ of the solution couple $(X_t, \mathcal{L}^1(X_t))$. However, to rigorously define the aforementioned operator on the Wasserstein space $\mathcal{P}_k(\mathbb{R}^d)\times\mathcal{P}_k(\mathcal{P}_k(\mathbb{R}^d))$, we must first guarantee that the distribution flow does not escape this space over time. This requires the solution to possess uniform moment bounds, which we establish below.

 \begin{lem}
 	\label{l1}
 	Let {\bf(A1)} and {\bf(A2)} hold with $p>l/2$. Then  the solution $(X_t)_{t\geq0}$ to the MV-SDEwCN \eqref{eq1.1} with initial value $\varsigma \in L^p(\Omega;\mathbb{R}^d)$
 	satisfies  $$
 	\sup_{t\geq 0}\mathbb{E}|X_t|^p\leq \mathbb{E}|\varsigma|^p+K,
 	$$ where $K$ is a positive constant independent of $\varsigma$.
 \end{lem}
 \begin{proof}
 It follows from Lemma 4.2 in \cite{M2008} that  
 	\begin{align}\label{in1}a^{p-2}b^2 \leq \frac{(p-2)\epsilon a^p}{p}+\frac{2 b^p}{p\epsilon^{(p-2)/2}},~~~~\forall~\epsilon, a, b>0 .
 	\end{align}
 	Utilizing the It\^o formula, \eqref{A2} and \eqref{in1}, for any $\lambda>0$ (defined latter), we have
 	\begin{align*}
 		&\mathrm{d}\left(e^{\lambda t}|X_t|^p\right)
 		\\&\leq e^{\lambda t}\Big(\lambda|X_t|^p+ \frac{p}{2}|X_t|^{p-2}\big(c_5-c_6|X_t|^2+c_7{W}_2^2(\mathcal{L}^1(X_t),\delta_{\bf 0})\big)\Big)\mathrm{d}t
 		\\&\quad+pe^{\lambda t}|X_t|^{p-2}X_t^Tg(X_t,\mathcal{L}^1(X_t))\mathrm{d}B_t
 		+pe^{\lambda t}|X_t|^{p-2}X_t^Tg^0(X_t,\mathcal{L}^1(X_t))\mathrm{d}B^0_t
 		\\
 		&\leq \frac{p}{2}e^{\lambda t}\Big(K+\Big(\frac{2\lambda}{p}
 		-\frac{  c_6+c_7}{2}\Big)|X_t|^{p}+c_7 |X_t|^{p-2}{W}_2^2(\mathcal{L}^1(X_t),\delta_{\bf 0})\Big)\mathrm{d}t
 		\\&\quad+pe^{\lambda t}|X_t|^{p-2}X_t^Tg(X_t,\mathcal{L}^1(X_t)) \mathrm{d}B_t
 		+pe^{\lambda t}|X_t|^{p-2}X_t^Tg^0(X_t,\mathcal{L}^1(X_t)) \mathrm{d}B^0_t,
 	\end{align*}
 	where 
 	$K=\frac{2c_5}{p} \left[\frac{2(p-2)c_5}{ p(c_6-c_7 )}\right]^{\frac{ p-2}{2}}.$
 	Taking the expectation of both sides, and applying the Fubini theorem, it follows that
 	\begin{align}\label{eq11}
 		e^{\lambda t}\mathbb{E}|X_t|^p
 		\leq \mathbb{E}|\varsigma|^p\!+\frac{p}{2}\!\int_0^t&e^{\lambda s}\Big(K+ \Big(\frac{2\lambda}{p}
 		- \frac{  c_6+c_7}{2}\Big)\mathbb{E}|X_s|^{p}
 		\nn\\&\quad+c_7 \mathbb{E}(|X_s|^{p-2}{W}_2^2(\mathcal{L}^1(X_s),\delta_{\bf 0}))\Big)\mathrm{d}s.
 	\end{align}
 	Making use of the H\"older inequality and \eqref{ls1}, one has \begin{align*} \begin{aligned}
 			\mathbb{E} (|X_s|^{p-2}{W}_2^2(\mathcal{L}^1(X_s),\delta_{\bf 0})) 
 			&\leq [\mathbb{E} |X_s|^{p} ]^{\frac{p-2}{p}}[\mathbb{E}{W}_2^p(\mathcal{L}^1(X_s),\delta_{\bf 0})]^{\frac{ 2}{p}}\\
 			&\leq  [\mathbb{E} |X_s|^{p} ]^{\frac{p-2}{p}}[\mathbb{E}{W}_p^p(\mathcal{L}^1(X_s),\delta_{\bf 0})]^{\frac{ 2}{p}} \leq\mathbb{E}|X_s|^p.\end{aligned}\end{align*}
 	This together with \eqref{eq11} that
 	\begin{align*}
 		e^{\lambda t}\mathbb{E}|X_t|^p
 		&\leq \mathbb{E}|\varsigma|^p+\frac{p}{2}\int_0^te^{ \lambda s}\Big(K+\Big(\frac{2\lambda}{p}
 		- \frac{  c_6-c_7}{2}\Big)\mathbb{E}|X_s|^{p}\Big)\mathrm{d}t.
 	\end{align*}
 	Let $\lambda={p(c_6-c_7)}/{4}>0$. Then,  
 	\begin{align*}
 		e^{\lambda t}\mathbb{E}|X_t|^p
 		&\leq \mathbb{E}|\varsigma|^p+\frac{pK}{2}\int_0^te^{\lambda s}
 		\mathrm{d}t\leq \mathbb{E}|\varsigma|^p+\frac{pK}{2\lambda}e^{\lambda t}.
 	\end{align*}
 	Dividing $e^{\lambda t}$ on both sides 
 	implies that the  desired assertion holds.
 \end{proof}
 
 We next define the operator. To proceed, we equip the state space with a suitable metric structure. For any $(u,\mu), (v,\vartheta)\in \mathcal{P}_p(\mathbb{R}^d)\times\mathcal{P}_p(\mathcal{P}_p(\mathbb{R}^d))$, define $$\boldsymbol{d}((u,\mu),(v,\vartheta))=W_p(u,v)+\mathbb{W}_p(\mu,\vartheta).$$ It is straightforward to verify that $\boldsymbol{d}$ is a metric. It is easy to obtain that  $(\mathcal{P}_p(\mathbb{R}^d)\times\mathcal{P}_p(\mathcal{P}_p(\mathbb{R}^d)),\boldsymbol{d})$ is a complete metric space.
Let $X^{s,\varsigma}_t$ be the solution to MV-SDEwCN \eqref{eq1.1} with any initial value $\varsigma\!\in\! L^p(\Omega;\mathbb{R}^d)$ at $s$. 
Define the  operator  $ \Phi^s_{t} :\mathcal{P}_p(\mathbb{R}^d)\times\mathcal{P}_p(\mathcal{P}_p(\mathbb{R}^d))\rightarrow \mathcal{P}_p(\mathbb{R}^d)\times\mathcal{P}_p(\mathcal{P}_p(\mathbb{R}^d))$ by
$$\Phi^{s  }_{t}(u_s^{s,\varsigma},\mu_s^{s,\varsigma})=(u_t^{s,\varsigma},\mu_t^{s,\varsigma}), ~~~~~\forall~t\geq s, $$ %
where 
$u_t^{s,\varsigma}=\mathcal{L}(X^{s,\varsigma}_t),\mu_t^{s,\varsigma}=\mathcal{L}(\mathcal{L}^1(X^{s,\varsigma}_t)) $ for any $t\geq s$. Write $\Phi _{t}=\Phi^{0  }_{t}, t\geq 0 $ for short. Owing to the weak uniqueness of MV-SDEwCN \eqref{eq1.1} proved in Lemma \ref{ls}, the family of operators $(\Phi^s_{t})_{t\geq s}$ is well defined. 
We now verify the semigroup property of $(\Phi^s_{t})$. 
\begin{lem} \label{lemma4.2}
	Let {\bf(A1)}-{\bf(A2)} hold with $p> l/2$. The family of operators $(\Phi^s_{t})_{t\geq s}$   forms a time-homogeneous semigroup, namely,  for $0\leq s\leq r\leq t$
	\begin{align*}
		\Phi^s_{t}(u^{s,\varsigma}_s,\mu^{s,\varsigma}_s)=\Phi^r_{t}(\Phi^s_{r}(u^{s,\varsigma}_s,\mu^{s,\varsigma}_s)) \quad{\rm and}\quad \Phi^s_{t}(u^{s,\varsigma}_s,\mu^{s,\varsigma}_s)=\Phi_{t-s}(u^{0,\varsigma }_0,\mu^{0,\varsigma}_0).
	\end{align*}
\end{lem}
\begin{proof} By MV-SDEwCN \eqref{eq1.1}, we have that for any $t\geq s$,
	\begin{align}\label{s317}
		\begin{aligned}
		X_t^{s,\varsigma}&=\varsigma +\int_s^tf(X^{s,\varsigma}_z ,\mathcal{L}^1(X^{s,\varsigma}_z))
		\mathrm{d} z+\int_s^tg(X^{s,\varsigma}_z ,\mathcal{L}^1(X^{s,\varsigma}_z))\mathrm{d} B_z
		\\&\quad+\int_s^t g^0(X^{s,\varsigma}_z,\mathcal{L}^1(X^{s,\varsigma}_z)) \mathrm{d} B_z^0.
		\end{aligned}
	\end{align}
	One observes that for any $r:~s\leq r\leq t$,
	\begin{align*}
		\begin{aligned}
			X_t^{s,\varsigma}&= X_r^{s,\varsigma}+\int_r^tf(X^{s,\varsigma}_z ,\mathcal{L}^1(X^{s,\varsigma}_z))
			\mathrm{d} z+\int_r^tg(X^{s,\varsigma}_z ,\mathcal{L}^1(X^{s,\varsigma}_z))\mathrm{d} B_z
			\\&\quad+\int_r^t g^0(X^{s,\varsigma}_z,\mathcal{L}^1(X^{s,\varsigma}_z)) \mathrm{d} B_z^0.
		\end{aligned}
	\end{align*}
	Therefore, $X_t^{s,\varsigma}$ is the  solution to MV-SDEwCN \eqref{eq1.1} with initial value $X_r^{s,\varsigma}\in L^p(\Omega;\mathbb{R}^d)$ at $r$. By Lemma \ref{l1}, the uniqueness of the solution to  MV-SDEwCN \eqref{eq1.1} implies that 
	$$X_t^{s,\varsigma}=X_t^{r,X_r^{s,\varsigma}}~~~ a.s.$$ Then $$\mathcal{L}(X_t^{s,\varsigma})=\mathcal{L}(X_t^{r,X_r^{s,\varsigma}}).$$
	By virtue of Lemma \ref{ls}, the conditional distribution of the  solution to MV-SDEwCN \eqref{eq1.1} satisfies the corresponding  SPDE \eqref{r3}. Thus, we obtain that for any $\varphi\in \mathcal{C}_c^\infty(\mathbb{R}^d;\mathbb{R})$, 
	$\mathbb{P}^0$-a.s.
	\begin{align}\label{s319}
		\begin{aligned}
		\langle \mathcal{L}^1(X^{s,\varsigma}_t),\varphi\rangle&=\langle \mathcal{L}^1(\varsigma),\varphi\rangle+\int_s^t\langle \mathcal{L}^1(X^{s,\varsigma}_z),L_{\mathcal{L}^1(X^{s,\varsigma}_z)} \varphi\rangle\mathrm{d}z
		\\&\quad+\int_s^t\langle \mathcal{L}^1(X^{s,\varsigma}_z),\varphi_xg^0(\cdot, \mathcal{L}^1(X^{s,\varsigma}_z))\rangle\mathrm{d}B_z^0.
		\end{aligned}
	\end{align}
	Then, for any $r: ~ s\leq r\leq t$, it follows that 
	\begin{align*}
		\langle \mathcal{L}^1(X^{s,\varsigma}_t),\varphi\rangle=&\langle \mathcal{L}^1(X^{s,\varsigma}_r),\varphi\rangle\!+\!\int_r^t\!\langle \mathcal{L}^1(X^{s,\varsigma}_z),L_{\mathcal{L}^1(X^{s,\varsigma}_z)} \varphi\rangle\mathrm{d}z\nonumber\\
		&\!+\!\int_r^t\!\langle \mathcal{L}^1(X^{s,\varsigma}_z),\varphi_xg^0(\cdot, \mathcal{L}^1(X^{s,\varsigma}_z))\rangle\mathrm{d}B_z^0.
	\end{align*}
	Due to the uniqueness of the solution to the SPDE \eqref{r3}, we conclude that $$\mathcal{L}^1(X_t^{s,\varsigma})=\mathcal{L}^1(X_t^{r,\mathcal{L}^1(X_r^{s,\varsigma})}) ~~~a.s.$$ for all $s\leq r\leq t$, which implies $\mathcal{L}(\mathcal{L}^1(X_t^{s,\varsigma}))=\mathcal{L}(\mathcal{L}^1(X_t^{r,\mathcal{L}^1(X_r^{s,\varsigma})}))$. Recalling the definition of $\Phi^s_t$, this yields $$\Phi^s_{t}(u^{s,\varsigma}_s,\mu^{s,\varsigma}_s)
	=\Phi^r_{t}(\mathcal{L}(X_r^{s,\varsigma}),\mathcal{L}(\mathcal{L}^1(X_r^{s,\varsigma})))
	=\Phi^r_{t}(\Phi^s_{r}(u^{s,\varsigma}_s,\mu^{s,\varsigma}_s)),$$ which establishes that  the family of operators $(\Phi^{s}_t)_{t \geq s}$ forms a semigroup. Next, we prove that the semigroup $(\Phi^{s}_t)_{t \geq s}$ is time-homogeneous. By a change of variables, we obtain from \eqref{s317} and \eqref{s319} that
	\begin{align*}
		\begin{aligned}
			&X_{t}^{s,\varsigma} =X_{s+{t-s}}^{s,\varsigma} =\varsigma+\int_0^{t-s}f(X^{s,\varsigma}_{s+r},\mathcal{L}^1(X^{s,\varsigma}_{s+r}))
			\mathrm{d} r+\int_0^{t-s}g(X^{s,\varsigma}_{s+r} ,\mathcal{L}^1(X^{s,\varsigma}_{s+r}))\mathrm{d} \bar{B}_r
			\\&\quad~~~~~~~~~~~ +\int_0^{t-s} g^0(X^{s,\varsigma}_{s+r},\mathcal{L}^1(X^{s,\varsigma}_{s+r})) \mathrm{d} \bar{B}_r^0, \\
			&\langle \mathcal{L}^1(X^{s,\varsigma}_{t}),\varphi\rangle=\langle \mathcal{L}^1(X^{s,\varsigma}_{s+(t-s)}),\varphi\rangle=\langle \mathcal{L}^1( \varsigma),\varphi\rangle\!+\!\int_0^{t-s}\langle \mathcal{L}^1(X^{s,\varsigma}_{s+r}),L_{\mathcal{L}^1(X^{s,\varsigma}_{s+r})} \varphi\rangle\mathrm{d}r
			\\&\quad~~~~~~~~~~~+ \int_0^{t-s}\langle \mathcal{L}^1(X^{s,\varsigma}_{s+r}),\varphi_xg^0(\cdot, \mathcal{L}^1(X^{s,\varsigma}_{s+r}))\rangle\mathrm{d}\bar{B}_r^0,~~~~\quad t\geq s,
		\end{aligned}
	\end{align*} 
	where $\bar{B}_r= B_{s+r}-B_s$ and $\bar{B}^0_r= B^0_{s+r}-B^0_s$. On the other hand, 
	we clearly have
	\begin{align*}
		\begin{aligned}
			&X_{t-s}^{0,\varsigma} =\varsigma+\int_0^{t-s}f(X^{0,\varsigma}_{r},\mathcal{L}^1(X^{0,\varsigma}_{r}))
			\mathrm{d} r+\int_0^{t-s}g(X^{0,\varsigma}_{r} ,\mathcal{L}^1(X^{0,\varsigma}_{r}))\mathrm{d}{B}_r
			\\&\quad~~~~~~~~~~~ +\int_0^{t-s} g^0(X^{0,\varsigma}_{r},\mathcal{L}^1(X^{0,\varsigma}_{r})) \mathrm{d} {B}_r^0, \\
			&\langle \mathcal{L}^1(X^{0,\varsigma}_{t-s}),\varphi\rangle=\langle \mathcal{L}^1( \varsigma),\varphi\rangle\!+\!\int_0^{t-s}\langle \mathcal{L}^1(X^{0,\varsigma}_{r}),L_{\mathcal{L}^1(X^{0,\varsigma}_{r})} \varphi\rangle\mathrm{d}r
			\\&\quad~~~~~~~~~~~+ \int_0^{t-s}\langle \mathcal{L}^1(X^{0,\varsigma}_{r}),\varphi_xg^0(\cdot, \mathcal{L}^1(X^{0,\varsigma}_{r}))\rangle\mathrm{d}{B}_r^0,~~~~\quad t\geq s,
		\end{aligned}
	\end{align*} 
	Since $(\bar{B}_r, \bar{B}^0_r)$ and $({B}_r, {B}_r^0)$ have the same distribution, applying Lemma \ref{ls} yields that for any $t\geq s$ 
	$$(\mathcal{L}(X_{t}^{s,\varsigma}),\mathcal{L}(\mathcal{L}^1(X_{t}^{s,\varsigma})))
	=(\mathcal{L}(X^{0,\varsigma}_{t-s}),\mathcal{L}(\mathcal{L}^1(X^{0,\varsigma}_{t-s}))).$$
	Thus,  
	for any $t\geq s$,  the desired assertion $\Phi^s_{t}(u^{s,\varsigma}_s,\mu^{s,\varsigma}_s)=
	\Phi_{t-s}(u^{0,\varsigma}_0,\mu^{0,\varsigma}_0)$ follows. The proof is complete.
\end{proof}

Owing to time-homogeneous property we restrict our attention to the semigroup $(\Phi_t)$ starting from $0$. Now we establish its asymptotic attractiveness. 
\begin{lem}\label{l3.6}
	Let  {\bf(A1)} and {\bf(A2)} hold with $p> l/2 $. For any $(u_0,\mu_0)$ and $(v_0,\vartheta_0)$ $ \in\mathcal{P}_p(\mathbb{R}^d)\times\mathcal{P}_p(\mathcal{P}_p(\mathbb{R}^d))$, 
	\begin{align*}
		\boldsymbol{d}\big(\Phi_t( u_0,\mu_0 ),\Phi_t ( v_0,\vartheta_0) \big)\leq 2\boldsymbol{d}((u_0,\mu_0),(v_0,\vartheta_0))e^{-\frac{ c_3-c_4}{2}t},\quad t\geq 0.
	\end{align*}
\end{lem}
\begin{proof}  Thanks to the existence of the optimal coupling (see, for example, Theorem 4.1 in \cite{V2009}), for any $(u_0,\mu_0)$,  $(v_0,\vartheta_0)\in\mathcal{P}_p(\mathbb{R}^d)\times
	\mathcal{P}_p(\mathcal{P}_p(\mathbb{R}^d))$,  there exists a pair of random variables $X_0$, $\bar{X}_0\in L^p(\Omega; \mathbb{R}^d)$ such that $\mathcal{L}(X_0)=u_0$, $\mathcal{L}(\mathcal{L}^1(X_0))=\mu_0$, $\mathcal{L}(\bar{X}_0)=v_0$, $\mathcal{L}(\mathcal{L}^1(\bar{X}_0))=\vartheta_0$, and \begin{align}\label{eq15}
		\begin{aligned}
		 \mathbb{W}^p_p(\mu_0,\vartheta_0)&=\mathbb{E}^0({W}_p^p(\mathcal{L}^1(X_0), \mathcal{L}^1(\bar{X}_0)))=
		\mathbb{E}|X_{0}\!-\!\bar{X}_{0}|^p,\\
		W^p_p(u_0,v_0)&\leq \mathbb{E}|X_0-\bar{X}_0|^p.
		\end{aligned}
	\end{align}
	Let $\{X_{t}\}_{t\geq 0}$ and $\{\bar{X}_{t}\}_{t\geq 0}$ be two solutions
	to \eqref{eq1.1} with the initial values $X_0$, $\bar{X}_0$, respectively. Utilizing the It\^o~formula and {\bf (A2)}, we arrive at
	\begin{align*}
		&\mathrm{d}e^{\lambda t}|X_{t}-\bar{X}_{t}|^p
		\\&\leq e^{\lambda t}\Big(\Big(\lambda-\frac{p c_3}{2}\Big)|X_{t}-\bar{X}_{t}|^p
		+\frac{pc_4}{2}|X_{t}-\bar{X}_{t}|^{p-2}{W}_2^2(\mathcal{L}^1(X_{t}),\mathcal{L}^1(\bar{X}_{t}))\Big)\mathrm{d}t
		\\&\quad+p|X_{t}-\bar{X}_{t}|^{p-2}( X_{t}-\bar{X}_{t})^T\big(g(X_{t},\mathcal{L}^1(X_{t})) -g(\bar{X}_{t},\mathcal{L}^1(\bar{X}_{t})) \big)\mathrm{d}B_t
		\\&\quad+p|X_{t}-\bar{X}_{t}|^{p-2}( X_{t}-\bar{X}_{t})^T\big(g^0(X_{t},\mathcal{L}^1(X_{t})) -g^0(\bar{X}_{t},\mathcal{L}^1(\bar{X}_{t})) \big)\mathrm{d}B^0_t,
	\end{align*}
	where $\lambda$ is a positive constant to be determined later. 
	Integrating both sides with respect to 
	$t$ and using the Fubini theorem, one has
	\begin{align*}
		e^{\lambda t}\mathbb{E}|X_{t}-\bar{X}_{t}|^p 
		&\leq \mathbb{E}|X_{0}-\bar{X}_{0}|^p+\int_0^te^{\lambda s}\Big(\Big(\lambda-\frac{p c_3}{2}\Big)\mathbb{E}|X_{s}-\bar{X}_{s}|^p
		\\&\quad+\frac{pc_4}{2}\mathbb{E}\big(|X_{s}-\bar{X}_{s}|^{p-2}{W}_2^2(\mathcal{L}^1(X_{s}),\mathcal{L}^1(\bar{X}_{s}))\big)\Big)\mathrm{d}s.
	\end{align*}
	By \eqref{ls1} and the H\"older inequality, we obtain that 
	\begin{align*}
		e^{\lambda t}\mathbb{E}|X_{t}-\bar{X}_{t}|^p
		\leq \mathbb{E}|X_{0}-\bar{X}_{0}|^p+\int_0^te^{\lambda s}\Big(\Big(\lambda-\frac{p (c_3-c_4)}{2}
		\Big)\mathbb{E}|X_{s}-\bar{X}_{s}|^p\Big)\mathrm{d}s.
	\end{align*}
	Let $\lambda={p (c_3-c_4)}/{2}$. 
	Then, dividing both sides of the above inequality by $e^{\lambda t}$, we obtain that 
	\begin{align}\label{eq18}
		\begin{aligned}
			\mathbb{E}|X_{t}-\bar{X}_{t}|^p
			&\leq \mathbb{E}|X_{0}-\bar{X}_{0}|^pe^{-\frac{p (c_3-c_4)t}{2}}.
		\end{aligned}
	\end{align}
	By the property of the Wasserstein distance, together with \eqref{eq15} and \eqref{eq18}, we observe that
	\begin{align}\label{s327}
		\begin{aligned}
			W_p(\mathcal{L}(X_t),\mathcal{L}(\bar{X}_t))
			\leq (\mathbb{E}|X_{t}-\bar{X}_{t}|^p)^{\frac{1}{p}} 
			\leq \mathbb{W}_p(\mu_0,\vartheta_0)e^{-\frac{(c_3-c_4)t}{2}}.
		\end{aligned}
	\end{align}
	Similarly,  we have
	\begin{align}\label{s328}
		\mathbb{W}_p(\mathcal{L}(\mathcal{L}^1(X_t)),\mathcal{L}(\mathcal{L}^1(\bar{X}_t))))\leq (\mathbb{E}|X_{t}-\bar{X}_{t}|^p)^{\frac{1}{p}} \leq \mathbb{W}_p(\mu_0,\vartheta_0)e^{-\frac{(c_3-c_4)t}{2}}.
	\end{align}
	Recalling the definition of the distance $\boldsymbol{d}$ on $(\mathcal{P}_k(\mathbb{R}^d)\times\mathcal{P}_k(\mathcal{P}_k(\mathbb{R}^d)))$, and using \eqref{s327} and \eqref{s328}, we conclude that
	\begin{align*}
		\boldsymbol{d}\big(\Phi_t (u_0,\mu_0 ),\Phi_t ( v_0,\vartheta_0 )\big)
		&\leq 2\mathbb{W}_p(\mu_0,\vartheta_0)e^{-\frac{(c_3-c_4)t}{2}}\\&\leq2\boldsymbol{d}\big((u_0,\mu_0), (v_0,\vartheta_0)\big)e^{-\frac{(c_3-c_4)t}{2}}.
	\end{align*}
	The desired assertion follows. The proof is  complete.
\end{proof}

The exponential contractivity established in Lemma \ref{l3.6} implies that the distribution flow admits at most one fixed point. Combining this with the tightness provided by the uniform moment bounds, we are now in a position to state our main result on the existence, uniqueness, and exponential stability of the invariant measure.
\begin{thm}\label{Th4.1}
	Let  {\bf(A1)} and {\bf(A2)} hold with $p> l/2$. The solution couple $(X_t, \mathcal{L}^1(X_t) ) $ to MV-SDEwCN \eqref{eq1.1} has a unique invariant measure $(u^*,\mu^*)\in\mathcal{P}_p(\mathbb{R}^d)$ $\times\mathcal{P}_p(\mathcal{P}_p
	(\mathbb{R}^d))$. Moreover, for any $(u_0,\mu_0)\in \mathcal{P}_p(\mathbb{R}^d) \times \mathcal{P}_p(\mathcal{P}_p(\mathbb{R}^d))$,
	\begin{align*}
		\boldsymbol{d}(\Phi_{t} (u_0,\mu_0 ), (u^*,\mu^*))
		\leq 2\boldsymbol{d}((u_0,\mu_0), (u^*,\mu^*))e^{-\frac{c_3-c_4}{2}t}, \quad \forall t\geq0.
	\end{align*}
\end{thm}
\begin{proof}  We first prove the existence of the invariant measure.  
	For $(\delta_{\bf 0},\delta_{\delta_{\bf 0}})\in\mathcal{P}_p(\mathbb{R}^d)\times\mathcal{P}_p
	(\mathcal{P}_p(\mathbb{R}^d))$, applying Lemma \ref{l3.6} and the H\"older inequality, we obtain
	\begin{align*}
		\begin{aligned}
			\boldsymbol{d}\big(\Phi_t (\delta_{\bf 0},\delta_{\delta_{\bf 0}}),\Phi_{s+t} (\delta_{\bf 0},\delta_{\delta_{\bf 0}})\big)&=\boldsymbol{d}\big(\Phi_t (\delta_{\bf 0},\delta_{\delta_{\bf 0}}),\Phi_{t} \Phi_{s}(\delta_{\bf 0},\delta_{\delta_{\bf 0}})\big)
			\\&\leq 2\boldsymbol{d}\big((\delta_{\bf 0},\delta_{\delta_{\bf 0}}), \Phi_{s}(\delta_{\bf 0},\delta_{\delta_{\bf 0}})\big)e^{-\frac{ c_3-c_4}{2}t}
			\\&\leq 4(\mathbb{E}|X^{0}_s|^p)^{\frac{1}{p}}e^{-\frac{ c_3-c_4}{2}t},
		\end{aligned}\end{align*}
	where $X^{0}_s$ denotes the solution of \eqref{eq1.1} with initial value ${\bf 0}\in\mathbb{R}^d$. 
	It follows from Lemma \ref{l1} that  
	\begin{align*}
		\lim_{t\rightarrow\infty}\sup_{s\geq0}\boldsymbol{d}\big(\Phi_t (\delta_{\bf 0},\delta_{\delta_{\bf 0}}),\Phi_{s+t} (\delta_{\bf 0},\delta_{\delta_{\bf 0}})\big)&\leq\lim_{t\rightarrow\infty}
		Ke^{-\frac{c_3-c_4}{2}t}=0,
	\end{align*}
	which implies that $\{\Phi_t(\delta_{\bf 0},\delta_{\delta_{\bf 0}})\}_{t\geq0}$ is a Cauchy family.  Owing to the completeness of  $\big(\mathcal{P}_p(\mathbb{R}^d)\times\mathcal{P}_p(\mathcal{P}_p(\mathbb{R}^d)),\boldsymbol{d}\big)$ we know that there exists a $(u^*,\mu^*)\in\mathcal{P}_p(\mathbb{R}^d)\times\mathcal{P}_p(\mathcal{P}_p(\mathbb{R}^d))$ such that
	\begin{align}\label{eq19}
		\lim_{t\rightarrow\infty}\boldsymbol{d}\big(\Phi_t (\delta_{\bf 0},\delta_{\delta_{\bf 0}}),(u^*,\mu^*)\big)=0.
	\end{align}
	We now prove that $(u^*,\mu^*) $ is invariant under $\Phi_t$. It follows from Lemmas \ref{lemma4.2} and \ref{l3.6} that for any $t, s\geq 0$,
	\begin{align*}
		\begin{aligned}
			&	\boldsymbol{d}(\Phi_{t}(u^*,\mu^*), (u^*,\mu^*))\\&\leq \boldsymbol{d}(\Phi_{t}(u^*,\mu^*), \Phi_{t}(\Phi_{s} (\delta_{\bf 0},\delta_{\delta_{\bf 0}}))) +\boldsymbol{d}(\Phi_{s}(\Phi_{t} (\delta_{\bf 0},\delta_{\delta_{\bf 0}})), \Phi_{s} (\delta_{\bf 0},\delta_{\delta_{\bf 0}}) )
			\\&\quad+\boldsymbol{d}(\Phi_{s}(\delta_{\bf 0},\delta_{\delta_{\bf 0}}), (u^*,\mu^*)) \\
			&\leq \boldsymbol{d}(\Phi_{s}(\delta_{\bf 0},\delta_{\delta_{\bf 0}}),  (u^*,\mu^*))\Big(1+2e^{-\frac{ c_3-c_4}{2}t}\Big)+ 2\boldsymbol{d}(\Phi_{t}(\delta_{\bf 0},\delta_{\delta_{\bf 0}}), (\delta_{\bf 0},\delta_{\delta_{\bf 0}}))e^{-\frac{ c_3-c_4}{2}s}\\&\leq 3\boldsymbol{d}(\Phi_{s}(\delta_{\bf 0},\delta_{\delta_{\bf 0}}),  (u^*,\mu^*))+ 4\sup_{t\geq 0} (\mathbb{E}|X^{0}_t|^p)^{\frac{1}{p}} e^{-\frac{ c_3-c_4}{2}s}.\end{aligned}
	\end{align*}
	Letting $s\rightarrow\infty$, and utilizing \eqref{eq19} and Lemma \ref{ls}, it follows that $$\boldsymbol{d}(\Phi_{t}(u^*,\mu^*), (u^*,\mu^*))=0$$ for any $t\geq 0$, which implies 
	$
	\Phi_{t}(u^*,\mu^*)\equiv (u^*,\mu^*). 
	$
	Next, we prove the uniqueness of the invariant measure. Assume that $(u^*_1,\mu^*_1)\in \mathcal{P}_p(\mathbb{R}^d)\times\mathcal{P}_p(\mathcal{P}_p(\mathbb{R}^d)) $ is also invariant to MV-SDEwCN \eqref{eq1.1}.  By Lemma \ref{l3.6}, we conclude that 
	\begin{align*}
		\boldsymbol{d}((u^*,\mu^*),(u^*_1,\mu^*_1))
		&=\boldsymbol{d}(\Phi_{t}(u^*,\mu^*),\Phi_{t}( u^*_1,\mu^*_1) )
		\\&\leq 2\boldsymbol{d}((u^*,\mu^*),(u^*_1,\mu^*_1))e^{-\frac{c_3-c_4}{2}t}, \quad \forall t\geq0.
	\end{align*}
	Taking the limit as $t\rightarrow\infty$ implies that $(u^*,\mu^*)=(u^*_1,\mu^*_1)$. Thus, $(u^*,\mu^*)\in\mathcal{P}_p(\mathbb{R}^d)\times\mathcal{P}_p(\mathcal{P}_p(\mathbb{R}^d))$ is the unique invariant measure under $\Phi_t$.  For any $(u_0,\mu_0)\in \mathcal{P}_p(\mathbb{R}^d)\times\mathcal{P}_p(\mathcal{P}_p(\mathbb{R}^d))$, we derive from Lemma \ref{l3.6} that  for any $ t\geq0$
	\begin{align*}
		\boldsymbol{d}(\Phi_{t}( u_0,\mu_0 ), (u^*,\mu^*)) 
		=\boldsymbol{d}(\Phi_{t} (u_0,\mu_0 ), \Phi_{t} (u^*,\mu^*) )
		\leq 2\boldsymbol{d}((u_0,\mu_0), (u^*,\mu^*))e^{-\frac{c_3-c_4}{2}t} .
	\end{align*}
	The proof is therefore complete.
\end{proof}

By Theorem \ref{Th4.1}, it is easy to obtain the invariant measure of the solution $X_t$ to MV-SDEwCN \eqref{eq1.1} as follows.
\begin{cor}\label{th3.2}
	Let  {\bf(A1)} and {\bf(A2)} hold with $p> l/2$. The solution to MV-SDEwCN \eqref{eq1.1}  has a unique invariant measure $ u^* \in\mathcal{P}_p(\mathbb{R}^d) $. Moreover, for any $X_0 \in L^p(\Omega;\mathbb{R}^d)$, the solution $X_t$ to MV-SDEwCN \eqref{eq1.1} has the property  
	\begin{align*} 
		W_p(\mathcal{L}(X_t),u^* )
		\leq Ke^{-\frac{c_3-c_4}{2}t}, \quad \forall t\geq0
	\end{align*} 
where $K$ is a constant depending on $(\mathbb{E}|X_0 |^p)^{\frac{1}{p}}$.
\end{cor}
\section{The strong law of large number}\label{Sect.3}
This section is devoted to the strong LLN of the MV-SDEswCN \eqref{eq1.1}, which reveals that the average of the system  in time tends to the spatial average.  To this end, we first estimate the moment error between the temporal average and the spatial average of the system. Subsequently, we prove that the temporal limit of this discrepancy converges to zero almost everywhere.

For this purpose, we first do some prepartion. Let $C_{lip}(\mathbb{R}^d;\mathbb{R})$ be the family of functions $\psi:\mathbb{R}^d\rightarrow \mathbb{R}$ satisfying the Lipschitz condition.
Let $C_{lip}(\mathbb{R}^d,\mathcal{P}_2(\mathbb{R}^d);\mathbb{R})$ be the family of functional $\Psi:\mathbb{R}^d\times\mathcal{P}_2(\mathbb{R}^d)\rightarrow \mathbb{R}$ satisfying that there exists a positive constants $c_{\Psi}$ such that 
\begin{align}
	\label{eq7.4}|\Psi(x,u)-\Psi(y,v)|\leq c_\Psi(|x-y|+W_2(u,v)).
\end{align}
This implies that 
\begin{align}\label{eq7.4*}
	|\Psi(x,u)|&\leq |\Psi(x,u)-\Psi(0,\delta_0)|+|\Psi(0,\delta_0)|\nn\\&\leq c_\Psi(|x|+W_2(u,\delta_0))+|\Psi(0,\delta_0)|\nn\\&\leq K(1+|x|+W_2(u,\delta_0)).
\end{align}
For any functional $\Psi\in C_{lip}(\mathbb{R}^d,\mathcal{P}_2(\mathbb{R}^d);\mathbb{R})$, we set 
\begin{align}\label{eq7.5}
	S_t(\Psi)=\int_0^t\Psi(X^{\varsigma}_s,\mathcal{L}^1(X^{\varsigma}_s))ds,\quad s_t(\Psi)=t^{-1}S_t(\Psi).
\end{align}
For the invariant measure $(u^*,\mu^*)$ of the MV-SDEwCN \eqref{eq1.1}, define 
$$\langle(u^*,\mu^*),\Psi\rangle =\int_{\mathcal{P}_2(\mathbb{R}^d)}\int_{\mathbb{R}^d}\Psi(x,u)u^*(dx)\mu^*(du), \quad \forall \Psi\in C_{lip}(\mathbb{R}^d,\mathcal{P}_2(\mathbb{R}^d);\mathbb{R}).$$
To estimate the moment error between the temporal average $s_t(\Psi)$ and the spatial average $\langle(u^*,\mu^*),\Psi\rangle$ of the MV-SDEwCN \eqref{eq1.1},  we now introduce an important decoupled equation as follows:
\begin{equation}\label{eq4.4}
	\begin{cases}
		\mathrm{d}Y_t=f(Y_t,u_t)\mathrm{d}t+g(Y_t,u_t)\mathrm{d}B_t+
		g^0(Y_t,u_t)\mathrm{d}B^0_t, \\
			\mathrm{d}\langle u_t,\varphi\rangle=\langle u_t,L_{u_t} \varphi\rangle\mathrm{d}t+\langle u_t,\varphi_x g^0(\cdot, u_t)\rangle\mathrm{d}B_t^0,\varphi\in \mathcal{C}_c^\infty(\mathbb{R}^d;\mathbb{R}),\quad t\geq0.
	\end{cases}
\end{equation}
By Theorem 2.1 in \cite{KNRS2022} and the proof of Lemma \ref{ls}, for any initial value $(\varsigma, u_0)\in L^p(\Omega;\mathbb{R}^d)\times L^p(\Omega^0;\mathcal{P}_p(\mathbb{R}^d))(p\geq l/2\vee 2)$, \eqref{eq4.4} exists unique solution $(Y^{\varsigma,u_0}_t,u_t^{u_0})$.  Let $(Y^{s,x,u}_t,u_t^{s,u})$ denote that the solution to \eqref{eq4.4} with initial value $(x,u)$ starting at time $s$.  we now prepare some useful lemmas.

\begin{lem}\label{l7.1}
		Let $\theta(x,u,\omega)$ be a scalar measurable random functional of $(x,u)\in\mathbb{R}^n\times \mathcal{P}(\mathbb{R}^n)$ independent of $\mathcal{F}_s$ such that for all $(x,u)\in\mathbb{R}^n\times \mathcal{P}(\mathbb{R}^n)$, $\mathbb{E}(|\theta(x,u,\omega)|)<\infty$. Let $(\zeta,\mathcal{L}^1(\zeta))$ be a couple of $\mathcal{F}_s$-measurable random variables. Assume   $\mathbb{E}(|\theta(\zeta,\mathcal{L}^1(\zeta),\omega)|)<\infty$. Then 
	\begin{align}\label{eq7.1}\mathbb{E}(\theta(\zeta,\mathcal{L}^1(\zeta),\omega)|\mathcal{F}_s)=\Theta(\zeta,\mathcal{L}^1(\zeta)),\end{align}
	where $\Theta(x,u)=\mathbb{E}(\theta(x,u,\omega))$.
\end{lem}
\begin{proof} Since the proof is rather technical, we divide it into two steps. \\
	{\bf Step 1.} We first prove that for a scalar bounded measurable random functional $\theta(x,u,\omega)$ independent of $\mathcal{F}_s$, the result holds. Following a similar procedure to that of Lemma 9.2 in \cite{M2008}, we assume that $\theta(x,u,\omega)$ has the following simple form 
	\begin{align}\label{eq7.2}
		\theta(x,u,\omega)=\sum_{i=1}^{k}\theta^{(1)}_i(x,u)\theta^{(2)}_i(\omega),
	\end{align}
	where $\theta^{(1)}_i(x,u)$ is a bounded deterministic function of $(x,u)$ and $\theta^{(2)}_i(\omega)$ is a bounded  measurable random variable independent of $\mathcal{F}_s$, then $\Theta(x,u)=\sum_{i=1}^{k}\theta^{(1)}_i(x,u)\mathbb{E}\theta^{(2)}_i(\omega)$.
	Now for any set $F\in\mathcal{F}_s$, we compute that
	$$
	\begin{aligned}
		\mathbb{E}\left(\theta(\zeta,\mathcal{L}^1(\zeta),\omega) I_F\right) & =\mathbb{E}\Big(\sum_{i=1}^{k}\theta^{(1)}_i(\zeta,\mathcal{L}^1(\zeta))\theta^{(2)}_i(\omega)I_F\Big)=\sum_{i=1}^k \mathbb{E}\Big(\theta^{(1)}_i(\zeta,\mathcal{L}^1(\zeta)) I_F\Big) \mathbb{E}\theta^{(2)}_i(\omega) \\
	&	=  \mathbb{E}\Big( \sum_{i=1}^k\theta^{(1)}_i(\zeta,\mathcal{L}^1(\zeta))\mathbb{E} \theta^{(2)}_i(\omega) I_F\Big)=\mathbb{E}\left(\Theta(\zeta,\mathcal{L}^1(\zeta)) I_F\right)
	\end{aligned}
	$$
	By definition, this means that \eqref{eq7.1} holds if $\theta(x,u,\omega)$ has the form of \eqref{eq7.2}. Since any bounded measurable random functional  $\theta(x,u,\omega)$ can be approximated by functions of the form \eqref{eq7.2}, the general result of the lemma for bounded measurable random functionals follow immediately.\\
	{\bf Step 2.} We now prove that the result holds for the scalar measurable random functional case. First let $\theta_{m}(x,u,\omega)=\theta(x,u,\omega)\wedge(m)\vee(-m)$, where $m$ is a positive constant. Then $|\theta_{m}(x,u,\omega)|\leq|\theta(x,u,\omega)|$ and $\lim_{m\rightarrow\infty}\theta_{m}(x,u,\omega)=\theta(x,u,\omega)$ a.s. By the property of the conditional expectation, we have
	$\mathbb{E}\left(\theta_{m}(\zeta,\mathcal{L}^1(\zeta),\omega) |\mathcal{F}_s\right)\leq\mathbb{E}\left(|\theta(\zeta,\mathcal{L}^1(\zeta),\omega) ||\mathcal{F}_s\right)<\infty, a.s.$ Then thanks to $\mathbb{E}(|\theta(x,u,\omega)|)<\infty$ and $\mathbb{E}(|\theta(\zeta,\mathcal{L}^1(\zeta),\omega)|)<\infty$,   it follows from the dominated convergence theorem
	\begin{align}
		\mathbb{E}(\theta(\zeta,\mathcal{L}^1(\zeta),\omega)|\mathcal{F}_s)&=\lim_{m\rightarrow\infty}\mathbb{E}(\theta_m(\zeta,\mathcal{L}^1(\zeta),\omega)|\mathcal{F}_s)\nn\\
		&=\lim_{m\rightarrow\infty}\Theta_m(\zeta,\mathcal{L}^1(\zeta))=\Theta(\zeta,\mathcal{L}^1(\zeta)),
	\end{align}
	where $\Theta_m(x,u)=\mathbb{E}(\theta_{m}(x,u,\omega)).$ This completes the proof.
\end{proof}

With the preceding lemma at hand, we are now ready to establish the conditional expectation estimate for the solution to the MV-SDEwCN. Specifically, the following theorem proves that the conditional expectation of a given functional of the solution couple decays at an exponential rate.
\begin{lem}\label{l7.2}
	Let  {\bf(A1)} and {\bf(A2)} hold with $p> l/2$. The solution couple $(X_t, \mathcal{L}^1(X_t) ) $ to MV-SDEwCN \eqref{eq1.1}  with initial value $\varsigma \in L^p(\Omega;\mathbb{R}^d)$ satisfies the following:
	for any functional $\Psi\in C_{lip}(\mathbb{R}^d,\mathcal{P}_2(\mathbb{R}^d);\mathbb{R})$ such that $\langle(u^*,\mu^*),\Psi\rangle=0$,
	\begin{align*}
		\mathbb{E}(\Psi(X_{t}^{\varsigma},\mathcal{L}^1(X_{t}^{\varsigma}))|\mathcal{F}_{s})\leq K(1+|X_{s}^{\varsigma}|+W_2(\mathcal{L}^1(X_{s}^{\varsigma}),\delta_0))e^{-\frac{(c_3-c_4)}{4}(t-s)},\quad \forall t\geq s.
	\end{align*}
\end{lem}	
	
\begin{proof}
	It follows from the existence and uniqueness of the strong solution to \eqref{eq1.1} and \eqref{r3} that
	$$(X_{t}^{\varsigma},\mathcal{L}^1(X_{t}^{\varsigma}))=(X_{t}^{s,X_{s}^{\varsigma}},\mathcal{L}^1(X_{t}^{s,X_{s}^{\varsigma}}))
	,\quad \forall t\geq s.$$
	By the uniqueness of the solution to \eqref{eq4.4}, we have for any  $t\geq s$,
	$$X_{t}^{s,X_{s}^{\varsigma}}=Y^{s,X_{s}^{\varsigma},\mathcal{L}^1(X_{s}^{\varsigma})}_t~ \mathrm{and}~\mathcal{L}^1(X_{t}^{s,X_{s}^{\varsigma}})
	=u_t^{s,\mathcal{L}^1(X_{s}^{\varsigma})},~~\mathrm{a.s}.$$
	It implies that
	\begin{align}\label{eq4.8}
		\mathbb{E}(\Psi(X_{t}^{\varsigma},\mathcal{L}^1(X_{t}^{\varsigma}))|\mathcal{F}_{s})=\mathbb{E}(\Psi(Y^{s,X_{s}^{\varsigma},\mathcal{L}^1(X_{s}^{\varsigma})}_t,u_t^{s,\mathcal{L}^1(X_{s}^{\varsigma})})|\mathcal{F}_{s}).
	\end{align}
	It is easy to see that $(Y^{s,x,u}_t,u_t^{s,u})$ is independent of $\mathcal{F}_s$ and $(X_{s}^{\varsigma},\mathcal{L}^1(X_{s}^{\varsigma}))$ is $\mathcal{F}_s$-measurable. Then using Lemma \ref{l7.1}  and \eqref{eq4.8} with $\theta(x,u,\omega)=\Psi(Y^{s,x,u}_t,u_t^{s,u})$, we obtain that
	\begin{align}\label{eq7.9}
		\mathbb{E}\big(\Psi(X_{t}^{\varsigma},\mathcal{L}^1(X_{t}^{\varsigma}))\big|\mathcal{F}_{s}\big)
=\mathbb{E}\big(\Psi(Y^{s,x,u}_t,u_t^{s,u})\big)\big|_{(x,u)=(X_{s}^{\varsigma},\mathcal{L}^1(X_{s}^{\varsigma}))}.
	\end{align}
	Making use of $\langle(u^*,\mu^*),\Psi\rangle=0$ and \eqref{eq7.4}, one obtains that
	\begin{align}\label{eq7.10}
		&\mathbb{E}\big[\Psi(Y^{s,x,u}_t,u_t^{s,u})\big]
=\langle\big(\mathcal{L}(Y^{s,x,u}_t),\mathcal{L}(u_t^{s,u})\big),\Psi\rangle-\langle(u^*,\mu^*),\Psi\rangle
		\nn\\&\!=\!\int_{\mathcal{P}(\mathbb{R}^d)}\!\int_{\mathbb{R}^d}\!\int_{\mathcal{P}(\mathbb{R}^d)}\!\int_{\mathbb{R}^d} \!\!\!\big(\Psi(x_1,u_1)\!-\!\Psi(x_2,u_2)\big)\mathcal{L}(Y^{s,x,u}_t)(dx_1)\mathcal{L}(u_t^{s,u})(du_1)u^*(dx_2)\mu^*(du_2)
		\nn\\&\!\leq \!c_{\Psi}\!\!\int_{\mathcal{P}(\mathbb{R}^d)}\!\int_{\mathbb{R}^d}\!\int_{\mathcal{P}(\mathbb{R}^d)}\!\int_{\mathbb{R}^d} \!\!\!\!(|x_1\!-\!x_2|\!+\!W_2(u_1,u_2))\mathcal{L}(Y^{s,x,u}_t)(dx_1)\mathcal{L}(u_t^{s,u})(du_1)u^*(dx_2)\mu^*(du_2).
	\end{align}
	Let $\varkappa\in L^2(\Omega;\mathbb{R}^d)$ be a random variable such that $\mathcal{L}(\varkappa)=u^*$ and $\mathcal{L}(\mathcal{L}^1(\varkappa))=\mu^*$. Then we have for any $t\geq s$, $(\mathcal{L}(X_{t}^{s,\varkappa}),\mathcal{L}(\mathcal{L}^1(X_{t}^{s,\varkappa})))=(u^*,\mu^*)$. This together with \eqref{eq7.10} gives that
	\begin{align}\label{eq7.11}
		\mathbb{E}\Psi(Y^{s,x,u}_t,u_t^{s,u})&\leq c_{\Psi}\big(\mathbb{E}|Y^{s,x,u}_t-X_{t}^{s,\varkappa}|+\mathbb{E}W_2(u_t^{s,u},\mathcal{L}^1(X_{t}^{s,\varkappa}))\big).
	\end{align}
	By the H\"older inequality, it is easy to see that
	\begin{align}\label{eq7.12}
		\mathbb{E}|Y^{s,x,u}_t-X_{t}^{s,\varkappa}|\leq \big(\mathbb{E}|Y^{s,x,u}_t-X_{t}^{s,\varkappa}|^2\big)^{\frac{1}{2}}
	\end{align}
	Now we turn to estimate $\mathbb{E}|Y^{s,x,u}_t-X_{t}^{s,\varkappa}|^2$. Making use of \eqref{eq1.1}, \eqref{eq4.4} and {\bf (A2)}, we have
		\begin{align}\label{eq4.13}
				e^{\frac{c_3-c_4}{2} t}\mathbb{E}|Y^{s,x,u}_t-X_{t}^{s,\varkappa}|^2
			&\leq \mathbb{E}|x-\varkappa|^2+\int_s^te^{\frac{c_3-c_4}{2} r}\Big(\Big(\frac{c_3-c_4}{2}-c_3
			\Big)\mathbb{E}|Y^{s,x,u}_r-X_{r}^{s,\varkappa}|^2\nn\\
			&\quad+c_4\mathbb{E}W_2^2(u_r^{s,u},\mathcal{L}^1(X_{r}^{s,\varkappa}))\big)\mathrm{d}r
			\nn\\&\leq \mathbb{E}|x-\varkappa|^2+\int_s^tc_4e^{\frac{c_3-c_4}{2} r}\mathbb{E}W_2^2(u_r^{s,u},\mathcal{L}^1(X_{r}^{s,\varkappa}))\mathrm{d}r.
	\end{align}
	For any $u\in\mathcal{P}_2(\mathbb{R}^d)$, there exists a random variable $\varsigma_u$ on $\Omega_2$ such that $\mathcal{L}(\varsigma_u)=u$.  Let $X_t^{s,\varsigma_u}$ be the MV-SDEwCN \eqref{eq1.1} with initial value $\varsigma_u$. By the relationship between \eqref{eq1.1} and \eqref{r3},  we have $u_t^{s,u}=\mathcal{L}^1(X_t^{s,\varsigma_u}),\forall t\geq s$. Then it follows from the H\"older inequality  and \eqref{ls1} that
	\begin{align}\label{eq7.13}
\mathbb{E}W_2^2(u_t^{s,u},\mathcal{L}^1(X_{t}^{s,\varkappa}))\leq \mathbb{E}|X_t^{s,\varsigma_u}-X_{t}^{s,\varkappa}|^2.
	\end{align}
	Similar to the proof of \eqref{eq18}, we have
	\begin{align}\label{eq7.14}
		\begin{aligned}
\mathbb{E}|X_t^{s,\varsigma_u}\!-\!X_{t}^{s,\varkappa}|^2
&\leq \mathbb{E}|\varsigma_u-\varkappa|^2e^{-(c_3-c_4)(t-s)}\leq K(1\!+\!W_2^2(u,\delta_0))e^{-(c_3-c_4)(t-s)}.\end{aligned}
	\end{align}
	Combining \eqref{eq4.13}-\eqref{eq7.14} gives that
	\begin{align}\label{eq4.16}
	e^{\frac{c_3-c_4}{2} t}\mathbb{E}|Y^{s,x,u}_t-X_{t}^{s,\varkappa}|^2
	&\leq \mathbb{E}|x-\varkappa|^2+K(1\!+\!W_2^2(u,\delta_0))e^{(c_3-c_4)s}\int_s^te^{-\frac{c_3-c_4}{2} r}\mathrm{d}r\nn\\
	&\leq K(1+|x|^2)+K(1\!+\!W_2^2(u,\delta_0))e^{\frac{c_3-c_4}{2} s}.
	\end{align}
	Dividing $e^{\frac{c_3-c_4}{2} t}$ on both side of \eqref{eq4.16}, we arrive at that
		\begin{align*}
		\mathbb{E}|Y^{s,x,u}_t-X_{t}^{s,\varkappa}|^2
		&\leq K(1+|x|^2+W_2^2(u,\delta_0))e^{-\frac{c_3-c_4}{2} (t-s)},
	\end{align*}
	which together with \eqref{eq7.12} implies that
	\begin{align}\label{eq4.17}
		\mathbb{E}|Y^{s,x,u}_t-X_{t}^{s,\varkappa}|
		&\leq K(1+|x|+W_2(u,\delta_0))e^{-\frac{c_3-c_4}{4} (t-s)}.
	\end{align}
	Using \eqref{eq7.13} and \eqref{eq7.14}, we have
	\begin{align}\label{eq4.18}
		\mathbb{E}W_2(u_t^{s,u},\mathcal{L}^1(X_{t}^{s,\varkappa}))\leq K(1\!+\!W_2(u,\delta_0))e^{-\frac{c_3-c_4}{2}(t-s)}.
	\end{align}
	Inserting \eqref{eq4.17} and \eqref{eq4.18} into \eqref{eq7.11} shows that
	\begin{align*}
		\mathbb{E}\Psi(Y^{s,x,u}_t,u_t^{s,u})&\leq K\big(1+|x|+W_2(u,\delta_0)\big)e^{-\frac{(c_3-c_4)}{4}(t-s)}.
	\end{align*}
	This together with \eqref{eq7.9} gives that
	\begin{align*}
		\mathbb{E}\big(\Psi(X_{t}^{\varsigma},\mathcal{L}^1(X_{t}^{\varsigma}))\big|\mathcal{F}_{s}\big)\leq K\big(1+|X_{s}^{\varsigma}|+W_2(\mathcal{L}^1(X_{s}^{\varsigma}),\delta_0)\big)e^{-\frac{(c_3-c_4)}{4}(t-s)}.
	\end{align*} 
	This completes the proof.
\end{proof}

 We now proceed to estimate the moment error between the temporal average $s_t(\Psi)$ and the spatial average $\langle(u^*,\mu^*),\Psi\rangle$ of the MV-SDEwCN \eqref{eq1.1}.
\begin{lem}\label{l7.3}
	Let  {\bf(A1)} and {\bf(A2)} hold with $p> l/2$. Let  $(u^*,\mu^*)$ be invariant measure of the solution couple $(X_t, \mathcal{L}^1(X_t) ) $ to MV-SDEwCN \eqref{eq1.1}  with initial value $\varsigma \in L^p(\Omega;\mathbb{R}^d)$. Then 
	for any functional $\Psi\in C_{lip}(\mathbb{R}^d,\mathcal{P}_2(\mathbb{R}^d);\mathbb{R})$ and any integer $1\leq q \leq p/2$, 
	\begin{align}	\mathbb{E}|s_t(\Psi)-\langle(u^*,\mu^*),\Psi\rangle|^{2q}\leq K(1+\mathbb{E}|\varsigma|^p)t^{-q}.
	\end{align}
\end{lem}	
\begin{proof}
	Without loss of generality, we may assume that $\langle(u^*,\mu^*),\Psi\rangle=0$ for a fixed functional 
$\Psi\in C_{lip}(\mathbb{R}^d,\mathcal{P}_2(\mathbb{R}^d);\mathbb{R})$. Define 
	\begin{align}
		I_q(t)=\sup_{0\leq r\leq t}\mathbb{E}|S_r(\Psi)|^{2q}.
	\end{align}
	By the definition of $S_r(\Psi)$ in \eqref{eq7.5}, we have
	\begin{align}
		&\mathbb{E}|S_r(\Psi)|^{2q}\nn\\
		=&\mathbb{E}	\int_{[0,r]^{2q}}\Psi(X_{r_{1}}^{\varsigma},\mathcal{L}^1(X_{r_{1}}^{\varsigma}))\Psi(X_{r_{2}}^{\varsigma},\mathcal{L}^1(X_{r_{2}}^{\varsigma}))\cdots\Psi(X_{r_{2q}}^{\varsigma},\mathcal{L}^1(X_{r_{2q}}^{\varsigma}))dr_{1}\cdots dr_{2q}\nn\\=&
		(2q)!\mathbb{E}	\int_{\Delta_{q}(r)}\Psi(X_{r_{1}}^{\varsigma},\mathcal{L}^1(X_{r_{1}}^{\varsigma}))\Psi(X_{r_{2}}^{\varsigma},\mathcal{L}^1(X_{r_{2}}^{\varsigma}))\cdots\Psi(X_{r_{2q}}^{\varsigma},\mathcal{L}^1(X_{r_{2q}}^{\varsigma}))dr_{1}\cdots dr_{2q}
	\end{align}
	where $\Delta_{q}(r)=\{(r_1,r_2,\cdots,r_{2q})\in\mathbb{R}^{2q}:0\leq r_1\leq \cdots\leq r_{2q}\leq r\}$.
	Making use of the property of conditional expectation, one arrives at that
	\begin{align*}
		\mathbb{E}|S_r(\Psi)|^{2q}=&(2q)!\mathbb{E}	\int_{\Delta_{q}(r)}\Psi(X_{r_{1}}^{\varsigma},\mathcal{L}^1(X_{r_{1}}^{\varsigma}))\Psi(X_{r_{2}}^{\varsigma},\mathcal{L}^1(X_{r_{2}}^{\varsigma}))\cdots
		\Psi(X_{r_{2q-2}}^{\varsigma},\mathcal{L}^1(X_{r_{2q-2}}^{\varsigma}))
		\nn\\&\times\Gamma(r_{2q-1},r_{2q})dr_{1}\cdots dr_{2q}\nn\\		
		=&(2q)!\mathbb{E}\int_{0}^r\int_{0}^{r_{2q}} \int_{\Delta_{q-1}(r_{2q-1})}\Psi(X_{r_{1}}^{\varsigma},\mathcal{L}^1(X_{r_{1}}^{\varsigma}))\cdots
		\Psi(X_{r_{2q-2}}^{\varsigma},\mathcal{L}^1(X_{r_{2q-2}}^{\varsigma}))
		\nn\\&\times\Gamma(r_{2q-1},r_{2q})dr_{1}\cdots dr_{2q}
		\nn\\		=&2q(2q-1)\mathbb{E}\int_{0}^r\int_{0}^{r_{2q}}\Gamma(r_{2q-1},r_{2q})
		|S_{r_{2q-1}}(\Psi)|^{2q-2}dr_{2q-1}	dr_{2q},
	\end{align*}
	where $$\Gamma(r_{2q-1},r_{2q})=\Psi(X_{r_{2q-1}}^{\varsigma},\mathcal{L}^1(X_{r_{2q-1}}^{\varsigma}))\mathbb{E}(\Psi(X_{r_{2q}}^{\varsigma},\mathcal{L}^1(X_{r_{2q}}^{\varsigma}))|\mathcal{F}_{r_{2q-1}}).$$
	Then it follows from the H\"older inequality that
	\begin{align*}
		\mathbb{E}|S_r(\Psi)|^{2q}\leq&2q(2q-1)\int_{0}^r\int_{0}^{r_{2q}}\big(\mathbb{E}|\Gamma(r_{2q-1},r_{2q})|^{q}\big)^{\frac{1}{q}} 
		\big(\mathbb{E}|S_{r_{2q-1}}(\Psi)|^{2q}\big)^{\frac{q-1}{q}}dr_{2q-1}	dr_{2q}.
	\end{align*}
	Now taking the supremum over $r \in [0, t]$, one obtains that
	\begin{align*}
		I_q(t)\leq&2q(2q-1)\int_{0}^t\int_{0}^{r_{2q}}\big(\mathbb{E}|\Gamma(r_{2q-1},r_{2q})|^{q}\big)^{\frac{1}{q}} 
		\big(\mathbb{E}|S_{r_{2q-1}}(\Psi)|^{2q}\big)^{\frac{q-1}{q}}dr_{2q-1}	dr_{2q}\nn\\
		\leq&2q(2q-1)(I_q(t))^{\frac{q-1}{q}}
		\int_{0}^t\int_{0}^{r_{2q}}\big(\mathbb{E}|\Gamma(r_{2q-1},r_{2q})|^{q}\big)^{\frac{1}{q}} dr_{2q-1}	dr_{2q}.
	\end{align*}
	Thus, we have 
	\begin{align}\label{eq7.9-}
		I_q(t)\leq&\Big(2q(2q-1)
		\int_{0}^t\int_{0}^{r_{2q}}\big(\mathbb{E}|\Gamma(r_{2q-1},r_{2q})|^{q}\big)^{\frac{1}{q}} dr_{2q-1}dr_{2q}\Big)^{q}.
	\end{align}
	We now estimate $\Gamma(r_{2q-1},r_{2q})$. By Lemma \ref{l7.2} one has that 
	\begin{align}\label{eq7.15}
		\mathbb{E}(\Psi(X_{r_{2q}}^{\varsigma},\mathcal{L}^1(X_{r_{2q}}^{\varsigma}))|\mathcal{F}_{r_{2q-1}})\leq K(1\!+\!|X_{r_{2q-1}}^{\varsigma}|\!+\!W_2(\mathcal{L}^1(X_{r_{2q-1}}^{\varsigma}),\delta_0))e^{-\frac{(c_3-c_4)(r_{2q}- r_{2q-1})}{4}}
	\end{align}
	Thus making use of \eqref{eq7.4*} and \eqref{eq7.15}, one can derive that
	\begin{align}\label{eq7.16}
		&\Gamma(r_{2q-1},r_{2q})=\Psi(X_{r_{2q-1}}^{\varsigma},\mathcal{L}^1(X_{r_{2q-1}}^{\varsigma}))\mathbb{E}(\Psi(X_{r_{2q}}^{\varsigma},\mathcal{L}^1(X_{r_{2q}}^{\varsigma}))|\mathcal{F}_{r_{2q-1}})\nn\\&\leq K|\Psi(X_{r_{2q-1}}^{\varsigma},\mathcal{L}^1(X_{r_{2q-1}}^{\varsigma}))|(1+|X_{r_{2q-1}}^{\varsigma}|+W_2(\mathcal{L}^1(X_{r_{2q-1}}^{\varsigma}),\delta_0))e^{-\frac{(c_3-c_4)(r_{2q}- r_{2q-1})}{4}}
		\nn\\&\leq K\big(1+|X_{r_{2q-1}}^{\varsigma}|+W_2(\mathcal{L}^1(X_{r_{2q-1}}^{\varsigma}),\delta_0)\big)^2e^{-\frac{(c_3-c_4)(r_{2q}- r_{2q-1})}{4}}
		\nn\\&\leq K\Big(1+|X_{r_{2q-1}}^{\varsigma}|^2+W_2^2(\mathcal{L}^1(X_{r_{2q-1}}^{\varsigma})),\delta_0)\Big)e^{-\frac{(c_3-c_4)(r_{2q}- r_{2q-1})}{4}}.
	\end{align}
	Inserting \eqref{eq7.16} to \eqref{eq7.9-}, we have that
	\begin{align}
		I_q(t)\leq&\Big(2q(2q-1)
		\int_{0}^t\int_{0}^{r_{2q}}\big(\mathbb{E}|\Gamma(r_{2q-1},r_{2q})|^{q}\big)^{\frac{1}{q}} dr_{2q-1}dr_{2q}\Big)^{q}
		\nn\\\leq &\Big(K\int_{0}^t\int_{0}^{r_{2q}}\big(1+(\mathbb{E}|X_{r_{2q-1}}^{\varsigma}|^{2q})^{\frac{1}{q}}\big)e^{-\frac{(c_3-c_4)(r_{2q}- r_{2q-1})}{4}}dr_{2q-1}dr_{2q}\Big)^{q}.
	\end{align}
	Then by Lemma \ref{l1}, one obtain that  
	\begin{align}\label{eq7.20}
		I_q(t)\leq &\Big(K\big(1+(\mathbb{E}|\varsigma|^p+K)^{\frac{1}{q}}\big)\int_{0}^t\int_{0}^{r_{2q}}e^{-\frac{(c_3-c_4)(r_{2q}- r_{2q-1})}{4}}dr_{2q-1}dr_{2q}\Big)^{q}\nn\\\leq &(K\big(1+(\mathbb{E}|\varsigma|^p+K)^{\frac{1}{q}}\big)t\Big)^{q}\leq K(1+\mathbb{E}|\varsigma|^p)t^q.
	\end{align}	 Then dividing $t^{2q}$ on both sides of \eqref{eq7.20} gives that
	\begin{align*}
		\mathbb{E}|s_t(\Psi)-\langle(u^*,\mu^*),\Psi\rangle|^{2q}\leq t^{-2q}I_q(t)\leq K(1+\mathbb{E}|\varsigma|^p)t^{-q}.
	\end{align*}	The proof is therefore complete.\end{proof}	

Equipped with the preceding estimates, we are now in a position to establish the main result of this section, namely the strong law of large number. 
\begin{thm}\label{t7.1}
	Let  {\bf(A1)} and {\bf(A2)} hold with $p> (l/2)\vee 4 $.  Let $(u^*,\mu^*)$ be the invariant measure  of the solution couple $(X_t, \mathcal{L}^1(X_t) ) $ to MV-SDEwCN \eqref{eq1.1}  with initial value $\varsigma \in L^p(\Omega;\mathbb{R}^d)$. Then for any functional $\Psi\in C_{lip}(\mathbb{R}^d,\mathcal{P}_2(\mathbb{R}^d);\mathbb{R})$, 
	\begin{align*}
		\frac{1}{T} \int_0^T \Psi\left(X_t^\varsigma,\mathcal{L}^1(X_t^\varsigma)\right) \mathrm{d} t \xrightarrow{\text { a.s. }} \langle(u^*,\mu^*),\Psi\rangle \quad  { \mathrm as }~ T \rightarrow \infty .
	\end{align*}
\end{thm}

\begin{proof}
To establish the desired assertion, it suffices to prove that 
\begin{align}\label{eq7.21}
	&\lim_{T \rightarrow \infty}\Big|	\frac{1}{T} \int_0^T (\Psi\left(X_t^\varsigma,\mathcal{L}^1(X_t^\varsigma)\right) -\langle(u^*,\mu^*),\Psi\rangle ) \mathrm{d} t\Big|=0, \text { a.s. }
\end{align}
Now we compute that
\begin{align}\label{eq7.22}
	&\Big|	\frac{1}{T} \int_0^T \big(\Psi(X_t^\varsigma,\mathcal{L}^1(X_t^\varsigma)) -\langle(u^*,\mu^*),\Psi\rangle\big) \mathrm{d} t\Big|\nn\\
	&\leq \Big|	\frac{1}{T} \int_0^{\lfloor T\rfloor} \!\!\!(\Psi\left(X_t^\varsigma,\mathcal{L}^1(X_t^\varsigma)\right) \!-\!\langle(u^*,\mu^*),\Psi\rangle ) \mathrm{d} t\Big|\!+\!\Big|	\frac{1}{T} \int_{\lfloor T\rfloor}^T\!\!\! (\Psi\left(X_t^\varsigma,\mathcal{L}^1(X_t^\varsigma)\right) \!-\!\langle(u^*,\mu^*),\Psi\rangle ) \mathrm{d} t\Big|\nn\\
	&\leq \Big|	\frac{1}{\lfloor T\rfloor} \int_0^{\lfloor T\rfloor} \!\!\!\!(\Psi\left(X_t^\varsigma,\mathcal{L}^1(X_t^\varsigma)\right)\! -\!\langle(u^*,\mu^*),\Psi\rangle ) \mathrm{d} t\Big|\!+\!\Big|	\frac{1}{T} \int_{\lfloor T\rfloor}^T\!\!\! (\Psi\left(X_t^\varsigma,\mathcal{L}^1(X_t^\varsigma)\right) \!-\!\langle(u^*,\mu^*),\Psi\rangle ) \mathrm{d} t\Big|\nn\\
	&=:I_1(T)+I_2(T).
\end{align}
Next, we verify that  $\lim_{T \rightarrow \infty}I_1(T)=0$, a.s. It is easy to see that for any $T>0$, $\lfloor T\rfloor$ is a positive integer. Thus we only need to prove that 
\begin{align}\label{eq7.23}
	\lim_{k \rightarrow \infty}\Big|	\frac{1}{k} \int_0^k (\Psi\left(X_t^\varsigma,\mathcal{L}^1(X_t^\varsigma)\right) -\langle(u^*,\mu^*),\Psi\rangle ) \mathrm{d} t\Big|=0, \text { a.s. }
\end{align}
To this end, for any $k\in N_+$ and $\delta\in(0,\frac{1}{4})$, define
\begin{align}
	\mathcal{A}_{k}^{\varsigma}=\Big\{\omega\in\Omega:\Big|	\frac{1}{k} \int_0^k (\Psi\left(X_t^\varsigma,\mathcal{L}^1(X_t^\varsigma)\right) -\langle(u^*,\mu^*),\Psi\rangle ) \mathrm{d} t\Big|>|c_{\Psi}|k^{-\delta}\Big\}.
\end{align}
By the Chebyshev inequality and Lemma \ref{l7.3}, we have that
\begin{align}
	\mathbb{P}(\mathcal{A}_{k}^{\varsigma})&\leq \frac{\mathbb{E}\Big|	\frac{1}{k} \int_0^k (\Psi\left(X_t^\varsigma,\mathcal{L}^1(X_t^\varsigma)\right) -\langle(u^*,\mu^*),\Psi\rangle ) \mathrm{d} t\Big|^4}{|c_{\Psi}|^4k^{-4\delta}}\nn\\&\leq  K (1+\mathbb{E}|\varsigma|^{4})k^{2(2\delta-1)}.
\end{align}
Since $\delta\in(0,1/4)$, we have $2(2\delta-1)<-1$. Then the series $\sum_{k=1}^{n}k^{2(2\delta-1)}$ converges as $n\rightarrow \infty$. It follows that
\begin{align}
	\sum_{k=1}^{\infty}\mathbb{P}(\mathcal{A}_{k}^{\varsigma})\leq K \mathbb{E}|\varsigma|^{4}\sum_{k=1}^{\infty}k^{2(2\delta-1)}\leq K \mathbb{E}|\varsigma|^{4}.
\end{align}
Define the random variable $\mathcal{B}^{\varsigma}_{\delta}(\omega)$ by $$\mathcal{B}^{\varsigma}_{\delta}(\omega):=\inf\Big\{j\in N_+:\Big|	\frac{1}{k} \int_0^k (\Psi\left(X_t^\varsigma,\mathcal{L}^1(X_t^\varsigma)\right) -\langle(u^*,\mu^*),\Psi\rangle ) \mathrm{d} t\Big|\leq|c_{\Psi}|k^{-\delta}, \mathrm{for}~k\geq j\!+\!1\Big\}.$$
Then making use of the Borel-Cantelli lemma, one has $\mathcal{B}^{\varsigma}_{\delta}(\omega)<\infty $ a.s., and 
\begin{align}
	\mathbb{P}\Big\{\omega\!\in\!\Omega:\Big|	\frac{1}{k} \int_0^k \!(\Psi\left(X_t^\varsigma,\mathcal{L}^1(X_t^\varsigma)\right)\!-\!\langle(u^*,\mu^*),\Psi\rangle ) \mathrm{d} t\Big|\leq |c_{\Psi}|k^{-\delta}, \mathrm{for}~k\!\geq\! \mathcal{B}^{\varsigma}_{\delta}(\omega)\!+\!1\Big\}\!=\!1,
\end{align}	 
which implies that \eqref{eq7.23} holds. We now  turn to prove that $\lim_{T \rightarrow \infty}I_2(T)=0$, a.s.  It is easy to see that
\begin{align}
	&I_2(T)=\Big|	\frac{1}{T} \int_{\lfloor T\rfloor}^{T} (\Psi\left(X_t^\varsigma,\mathcal{L}^1(X_t^\varsigma)\right) -\langle(u^*,\mu^*),\Psi\rangle ) \mathrm{d} t\Big|\nn\\
	&\leq
	\Big|	\frac{1}{\lfloor T\rfloor} \int_{\lfloor T\rfloor}^{T} (\Psi\left(X_t^\varsigma,\mathcal{L}^1(X_t^\varsigma)\right) -\langle(u^*,\mu^*),\Psi\rangle ) \mathrm{d} t\Big|
	\nn\\
	&\leq	\frac{1}{\lfloor T\rfloor} \int_{\lfloor T\rfloor}^{\lceil T \rceil
	} |\Psi\left(X_t^\varsigma,\mathcal{L}^1(X_t^\varsigma)\right) -\langle(u^*,\mu^*),\Psi\rangle|  \mathrm{d} t.
\end{align}
Note that for any $T>0$, there exists a positive integer $k$ such that $k\leq T<(k+1)$. Consequently, we only need to prove that 
\begin{align}\label{eq7.29}
	\lim_{k\rightarrow \infty}	\frac{1}{k} \int_{k}^{k+1
	} |\Psi\left(X_t^\varsigma,\mathcal{L}^1(X_t^\varsigma)\right) -\langle(u^*,\mu^*),\Psi\rangle|  \mathrm{d} t=0, ~\mathrm{a.s.}
\end{align}
For any $k\in N_+$ and $\delta\in(0,\frac{1}{4})$, define
\begin{align}
	\mathcal{C}_{k}^{\varsigma}=\Big\{\omega\in\Omega:	\frac{1}{k} \int_k^{k+1} |\Psi\left(X_t^\varsigma,\mathcal{L}^1(X_t^\varsigma)\right) -\langle(u^*,\mu^*),\Psi\rangle | \mathrm{d} t>|c_{\Psi}|k^{-\delta}\Big\}.
\end{align}
By the Chebyshev inequality, \eqref{eq7.4*} and Lemma \ref{l1}, we have that
\begin{align}
	\mathbb{P}(\mathcal{C}_{k}^{\varsigma})&\leq \frac{\mathbb{E}\Big|	\frac{1}{k} \int_k^{k+1}| \Psi\left(X_t^\varsigma,\mathcal{L}^1(X_t^\varsigma)\right) -\langle(u^*,\mu^*),\Psi\rangle | \mathrm{d} t\Big|^2}{|c_{\Psi}|^2k^{-2\delta}}\nn\\&\leq  
	\frac{\mathbb{E}\Big|	\frac{1}{k} \Big(K+K\int_k^{k+1} (1+|X_t^\varsigma|+W_2(\mathcal{L}^1(X_t^\varsigma),\delta_0))\mathrm{d} t\Big|^4}{|c_{\Psi}|^2k^{-2\delta}}\nn\\&\leq K (\mathbb{E}|\varsigma|^{2}+K)k^{2(\delta-1)}.
\end{align}
It follows from $\delta\in(0,1/4)$ that $2(\delta-1)<-1$. Then the series $\sum_{k=1}^{n}k^{2(\delta-1)}$ converges as $n\rightarrow \infty$, which implies that
\begin{align}
	\sum_{k=1}^{\infty}\mathbb{P}(\mathcal{C}_{k}^{\varsigma})\leq K (\mathbb{E}|\varsigma|^{2}+K)\sum_{k=1}^{\infty}k^{2(\delta-1)}< \infty.
\end{align}
Define the random variable $\mathcal{D}^{\varsigma}_{\delta}(\omega)$ by $$\mathcal{D}^{\varsigma}_{\delta}(\omega):=\inf\Big\{j\in N_+:	\frac{1}{k} \int_k^{k+1} |(\Psi\left(X_t^\varsigma,\mathcal{L}^1(X_t^\varsigma)\right) -\langle(u^*,\mu^*),\Psi\rangle )| \mathrm{d} t\leq|c_{\Psi}|k^{-\delta}, \mathrm{for}~k\geq j\!+\!1\Big\}.$$
Then by the Borel-Cantelli lemma, we arrive at that $\mathcal{D}^{\varsigma}_{\delta}(\omega)<\infty $ a.s., and 
\begin{align}
	\mathbb{P}\Big\{\omega\!\in\!\Omega:\frac{1}{k} \int_k^{k+1} |(\Psi\left(X_t^\varsigma,\mathcal{L}^1(X_t^\varsigma)\right) -\langle(u^*,\mu^*),\Psi\rangle )| \mathrm{d} t\leq |c_{\Psi}|k^{-\delta}, \mathrm{for}~k\!\geq\! \mathcal{D}^{\varsigma}_{\delta}(\omega)\!+\!1\Big\}\!=\!1.
\end{align}	 
This implies that \eqref{eq7.29} holds. One thus has $\lim_{T\rightarrow \infty}I_2(T)=0$, a.s. Now  taking limit with respect to T on both sides of \eqref{eq7.22} gives \eqref{eq7.21}. The proof is therefore complete.
\end{proof}

Specializing Theorem \ref{t7.1} to the case where $\Psi(x,u)=\psi(x)$ gives the following corollary, which establishes the ergodicity of the invariant measure $u^*$ for the solution $X_t^{\varsigma}$ to the MV-SDEwCN \eqref{eq1.1}.
\begin{cor}
	Let  {\bf(A1)} and {\bf(A2)} hold with $p> l/2\vee 4$. The unique invariant measure $ u^* \in\mathcal{P}_p(\mathbb{R}^d) $ to the MV-SDEwCN \eqref{eq1.1} satisfies that for any function $\psi\in C_{lip}(\mathbb{R}^d;\mathbb{R})$, 
	$$
	\begin{aligned}
		\frac{1}{T} \int_0^T \Psi\left(X_t^\varsigma\right) \mathrm{d} t \xrightarrow{\text { a.s. }} \langle u^*,\Psi\rangle \quad \text { as } T \rightarrow \infty .
	\end{aligned}
	$$
\end{cor}

\section{Uniform-in-time conditional propagation
	of chaos}\label{Sect.4}
To further characterize the asymptotic relationship between the IPSwCN and the MV-SDEwCN, this section focuses on the uniform-in-time conditional propagation of chaos. Additionally, we prove that the distribution of a single particle converges to $u^*$ while the distribution of the empirical measure of the IPSwCN converges to $\mu^*$.
For the solution $\{{X}^{i}_t\}_{i\geq 1}$ of \eqref{eq5.2}, as proved in \cite{KNRS2022}, it holds that,for any $i\geq 1$,
\begin{align}\label{eq5.2*}
	\mathbb{P}^0(\mathcal{L}^1({X}^{i}_t)=\mathcal{L}^1({X}^{1}_t), \forall t\geq 0)=1.
\end{align}
Furthermore by the identical distribution of $X^{i,N}_t,i=1,\cdots,N$ and the similar argument as the proof of Lemma \ref{l1}, one obtains that the uniform moment boundedness of $X^{i,N}_t,i=1,\cdots,N$, namely,
\begin{align*}
	\sup_{t\geq 0}\mathbb{E}|X^{i,N}_t|^p\leq \mathbb{E}|\varsigma|^p+K,
	\end{align*}
	 where $K$ is a positive constant independent of $\varsigma$.
Now we establish the uniform-in-time propagation of chaos.
\begin{thm}\label{T2}
	Let {\bf (A1)} and {\bf (A2)} hold with $p>l/2$.
	Then, for any $2\leq q< p $, the solutions $X^{i,N}_t$ to IPSwCN \eqref{eq5.1} and  ${X}^i_t$ to non-IPSwCN \eqref{eq5.2} have the properties
	\begin{align}
		\label{eq5.4}&\max_{1\leq i\leq N}\sup_{t\geq0}\mathbb{E}|{X}^i_t -X^{i,N}_t|^{q}\leq K\varepsilon_{N}, ~~~~~~\sup_{t\geq0}\mathbb{E}{W}_{q}^{q}\big(\mathcal{L}^1({X}^{i}_t), {u_t^N} \big)\leq K\varepsilon_{N},
	\end{align}
	where $u_t^N=\frac{1}{N}\sum_{j=1}^N\delta_{X^{j,N}_t}$, $K$ is a positive constant independent of  $N$  and 
	\begin{align}\label{L**}
		\varepsilon_N= \begin{cases}N^{-1 / 2}+N^{-(p-q) / p} & \text { if } q>d / 2 \text { and } p \neq 2 q, \\
			N^{-1 / 2} \log (1+N)+N^{-(p-q) / p} & \text { if } q=d / 2 \text { and } p \neq 2 q, \\
			N^{-q / d}+N^{-(p-q) / p} & \text { if } q\in(0, d / 2) \text { and } p \neq d /(d-q) .\end{cases}
	\end{align}
\end{thm}
\begin{proof}
	Fix $2\leq q<p$ and $1\leq i\leq N$. By \eqref{eq5.1} and \eqref{eq5.2}, we obtain that
	\begin{align*}
		\mathrm{d}({X}^{i}_t-X^{i,N}_t)&= \big(f({X}^{i}_t,\mathcal{L}^1({X}^{i}_t))-f(X^{i,N}_t, u_t^N)\big)\mathrm{d} t
		\\&\quad+\big(g({X}^{i}_t,\mathcal{L}^1({X}^{i}_t))-g(X^{i,N}_t, u_t^N)\big)\mathrm{d}B^i_t
		\\&\quad
		+\big(g^0({X}^{i}_t,\mathcal{L}^1({X}^{i}_t))-g^0(X^{i,N}_t, u_t^N)\big)\mathrm{d}B^{0}_t.
	\end{align*}
	Applying the It\^o formula  gives that
	\begin{align*}
		&\mathrm{d}\Big(e^{\lambda t}|{X}^{i}_t-X^{i,N}_t|^q\Big)\leq  e^{\lambda t}\Big(\lambda|{X}^{i}_t-X^{i,N}_t|^q
		\\&+\frac{q}{2}|{X}^{i}_t\!-\!X^{i,N}_t|^{q-2}\Big(2( {X}^{i}_t-X^{i,N}_t)^T(f({X}^{i}_t,\mathcal{L}^1({X}^{i}_t)) -f(X^{i,N}_t, u_t^N))
		\\&+(q\!-\!1)\big(|g({X}^{i}_t,\mathcal{L}^1({X}^{i}_t)) \!-\!g(X^{i,N}_t, u_t^N))|^2
		\!+\!|g^0({X}^{i}_t,\mathcal{L}^1({X}^{i}_t)) \!-\!g^0(X^{i,N}_t, u_t^N))|^2\big)\Big)\Big)\mathrm{d}t
		\\&+qe^{\lambda t}|{X}^{i}_t-X^{i,N}_t|^{q-2}( {X}^{i}_t-X^{i,N}_t)^T(g({X}^{i}_t,\mathcal{L}^1({X}^{i}_t)) -g(X^{i,N}_t, u_t^N)) \mathrm{d}B_t
		\\&+qe^{\lambda t}|{X}^{i}_t-X^{i,N}_t|^{q-2}( {X}^{i}_t-X^{i,N}_t)^T(g^0({X}^{i}_t,\mathcal{L}^1({X}^{i}_t)) -g^0(X^{i,N}_t, u_t^N)) \mathrm{d}B^0_t,
	\end{align*}
	where $\lambda>0$ is to be determined later.
	Taking expectations on both sides and using {\bf (A2)}, we have 
	\begin{align}\label{eq5.6}
		\begin{aligned}
			&e^{\lambda t}\mathbb{E}|{X}^{i}_t\!-\!X^{i,N}_t|^q
			\\&\leq \mathbb{E}\!\int_0^t\!\!\!e^{\lambda s}\Big(\Big(\lambda\!-\!\frac{qc_3}{2} \Big)|{X}^{i}_s\!-\!X^{i,N}_s|^q\!+\!\frac{qc_4}{2}|{X}^{i}_s\!-\!X^{i,N}_s|^{q-2}{W}_2^2(\mathcal{L}^1({X}^{i}_s), u_s^N)\Big)\mathrm{d}s
			.
		\end{aligned}
	\end{align}
	By virtue of the triangle inequality of the Wasserstein distance and the inequality 
	$$2ab\leq \frac{2c_4}{c_3-c_4} a^2+\frac{c_3-c_4}{2c_4}b^2,\forall a,b\in\mathbb{R},$$ it follows that
	\begin{align}\label{eq5.62}
		\begin{aligned}
			&{W}^2_2(\mathcal{L}^1({X}^{i}_s), u_s^N)
			\\&\leq\big({W}_{2}(\mathcal{L}^1({X}^{i}_s), v_s^N)+{W}_{2}( v_s^N, u_s^N)\big)^2
			\\&= {W}^2_{2}(\mathcal{L}^1({X}^{i}_s)), v_s^N)+2{W}_{2}(\mathcal{L}^1({X}^{i}_s), v_s^N){W}_{2}( v_s^N, u_s^N)+{W}^2_{2}( v_s^N, u_s^N)
			\\&\leq \frac{c_3+c_4}{c_3-c_4}{W}^2_{2}(\mathcal{L}^1({X}^{i}_s)), v_s^N)+\frac{c_3+c_4}{2c_4}{W}^2_{2}( v_s^N, u_s^N),
		\end{aligned}
	\end{align}
	where 
	$v_s^N=\frac{1}{N}\sum_{j=1}^N\delta_{{X}^{j}_s}.$
	Inserting \eqref{eq5.62} into \eqref{eq5.6} 
implies that
	\begin{align}\label{eq5.7}
		\begin{aligned}
		&e^{\lambda t}\mathbb{E}|{X}^{i}_t-X^{i,N}_t|^q
		\\&
		\leq \mathbb{E}\int_0^te^{\lambda s}\Big(\Big(\lambda-\frac{qc_3}{2} \Big)|{X}^{i}_s-X^{i,N}_s|^q+\frac{qc_4}{2}|{X}^{i}_s-X^{i,N}_s|^{q-2}
		\\&\quad\times\Big(\frac{c_3+c_4}{c_3-c_4}{W}^2_{2}(\mathcal{L}^1({X}^{i}_s)), v_s^N)+\frac{c_3+c_4}{2c_4}{W}^2_{2}( v_s^N, u_s^N)\Big)
		\\&= \mathbb{E}\int_0^t\!e^{\lambda s}\Big(\Big(\lambda\!-\!\frac{qc_3}{2} \Big)|{X}^{i}_s\!-\!X^{i,N}_s|^q
	\\&\quad+\!\frac{qc_4(c_3\!+\!c_4)}{2(c_3\!-\!c_4)}|{X}^{i}_s\!-\!X^{i,N}_s|^{q-2}{W}_{2}^2(\mathcal{L}^1({X}^{i}_s), v_s^N)
		\\&\quad+\frac{q(c_3+c_4)}{4}|{X}^{i}_s-X^{i,N}_s|^{q-2}{W}_{2}^2( v_s^N, u_s^N)\Big)\mathrm{d}s
		.
		\end{aligned}
	\end{align} 
	By the Young inequality and the properties of the Wasserstein distance, we obtain that
	\begin{align}\label{eq5.8}
		\begin{aligned}
		&\frac{qc_4(c_3+c_4)}{2(c_3-c_4)}|{X}^{i}_s-X^{i,N}_s|^{q-2}{W}_{2}^2(\mathcal{L}^1({X}^{i}_s), v_s^N)
		\\&\leq\frac{q(c_3-c_4)}{8}|{X}^{i}_s-X^{i,N}_s|^{q}+K{W}_{q}^q(\mathcal{L}^1({X}^{i}_s), v_s^N),
		\end{aligned}
	\end{align}
	and 
	\begin{align}\label{eq5.9}
		&\frac{q(c_3+c_4)}{4}|{X}^{i}_s-X^{i,N}_s|^{q-2}{W}_{2}^2( v_s^N, u_s^N)
		\nonumber\\&\leq \frac{(q-2)(c_3+c_4)}{4}|{X}^{i}_s-X^{i,N}_s|^{q}+\frac{(c_3+c_4)}{2}{W}_{q}^q( v_s^N, u_s^N).
	\end{align}
	Substituting \eqref{eq5.8} and \eqref{eq5.9} into \eqref{eq5.7} and applying the Fubini theorem, we deduce that 
	\begin{align}\label{eq5.10}
		e^{\lambda t}\mathbb{E}|{X}^{i}_t-X^{i,N}_t|^q
		&\leq \int_0^te^{\lambda s}\Big(\Big(\lambda-\frac{(q+4)c_3-(q-4)c_4}{8} \Big)\mathbb{E}|{X}^{i}_s-X^{i,N}_s|^q
		\nonumber\\&\quad+K\mathbb{E}{W}_{q}^q(\mathcal{L}^1({X}^{i}_s), v_s^N)
		+\frac{(c_3+c_4)}{2}\mathbb{E}{W}_{q}^q( v_s^N, u_s^N)\Big)\mathrm{d}s.
	\end{align}
	We now analyse the term $\mathbb{E}{W}_{q}^q(\mathcal{L}^1({X}^{i}_s), v_s^N)$ in the right side of \eqref{eq5.10}. It can be easily deduce from \eqref{eq5.2*} that there exists a set $\tilde{\Omega}^0$ satisfying $\mathbb{P}^0(\tilde{\Omega}^0)=1$ such that for any $\omega^0\in\tilde{\Omega}^0$,
	$$\mathcal{L}^1({X}^{i}_s)(\omega^0)=\mathcal{L}^1({X}^{1}_s)(\omega^0),\forall s\geq 0.$$
	Namely, for any $\omega^0\in\tilde{\Omega}^0$ and $t\geq 0$, $ {X}^{i}_s(\omega^0,\cdot),   i=1,\cdots,   N, $ are independent and identically distributed random variables with distribution $\mathcal{L}^1({X}^{1}_s)(\omega^0)$. Using the Fubini theorem and the H\"older inequality, one deduces from  Theorem 1 in \cite{FG2015} that for any $s\geq 0$ and any $\omega^0\in\tilde{\Omega}^0$
		\begin{align*}
			\mathbb{E}^1{W}_{q}^{q}\big(\mathcal{L}^1({X}^{i}_s), v_s^N\big) 
			&=\mathbb{E}^1 {W}_{q}^{q}\Big(\mathcal{L}^1({X}^{i}_s)(\omega^0),\frac{1}{N}\sum_{j=1}^N\delta_{{X}^{j}_s(\omega^0,\cdot)}\Big)
			\\&\leq K\varepsilon_{N} \mathbb{E}^1{W}_{p}^{q}(\mathcal{L}^1({X}^{i}_s)(\omega^0),\delta_{\bf 0}),
	\end{align*}
	where $K$ is a positive constant depending only on $d,p,q$, and $\varepsilon_{N}$ is defined by \eqref{L**}. 
	Taking expectations on both sides of the above inequality and using Lemma \ref{l1}, we have 
	\begin{align}\label{eq5.11}
		\mathbb{E}{W}_{q}^{q}\big(\mathcal{L}^1({X}^{i}_s), v_s^N\big) 
		\leq  K\varepsilon_{N} \mathbb{E} {W}_{p}^{q}(\mathcal{L}^1({X}^{i}_s),\delta_{\bf 0})\leq K\varepsilon_{N}(\mathbb{E}|{X}^1_s|^p)^{\frac{q}{p}}. 
	\end{align}
	{As discussed in the proof of Theorem 2.12 in \cite{CD2018b}, $(X_s^j,X_s^{j,N}),  j=1,\cdots,  N,$ are identically
		distributed. Noticing that $\frac{1}{N}\sum_{i=1}^{N}\delta_{({X}^{j}_s,{X}^{j,N}_s)}$ is a coupling of $v_s^N$ and $u_s^N$,  we obtain   
		\begin{align}\label{eq5.12}
			\mathbb{E}{W}_{q}^{q}( v_s^N, u_s^N)&
			\leq\mathbb{E}\Big[\frac{1}{N}\sum_{j=1}^N|{X}^{j}_s-{X}^{j,N}_s|^{q}\Big]
			=\mathbb{E}|{X}^{i}_s-{X}^{i,N}_s|^{q}.
	\end{align}}
	Inserting \eqref{eq5.11} and \eqref{eq5.12} into \eqref{eq5.10} yields that
	\begin{align*}
		e^{\lambda t}\mathbb{E}|{X}^{i}_t-X^{i,N}_t|^q
		&\leq \int_0^te^{\lambda s}\Big(K \varepsilon_{N} +\Big(\lambda-\frac{q(c_3-c_4)}{8} \Big)\mathbb{E}|{X}^{i}_s-X^{i,N}_s|^q\Big)\mathrm{d}s.
	\end{align*}
	Let $\lambda={q(c_3-c_4)}/{8} >0$.  
	Dividing $e^{\lambda t}$ on both sides of the above inequality yields that
	\begin{align}\label{eq5.14}
		\mathbb{E}|{X}^{i}_t-X^{i,N}_t|^q
		&\leq K\varepsilon_{N},~~~~\forall~t\geq 0, 
	\end{align}
	which establishes the desired assertion \eqref{eq5.4}. Furthermore, combining \eqref{eq5.11}-\eqref{eq5.14} leads to
	\begin{align*}
		\mathbb{E}{W}_{q}^{q}(\mathcal{L}^1({X}^{i}_t), u_t^N)
		&\leq \mathbb{E}\big({W}_{q}(\mathcal{L}^1({X}^{i}_t), v_t^N)+{W}_{q}( v_t^N, u_t^N)\big)^q
		\nonumber\\&\leq 2^{q-1}\mathbb{E}{W}_{q}^q(\mathcal{L}^1({X}^{i}_t), v_t^N)+2^{q-1}\mathbb{E}{W}_{q}^{q}( v_t^N, u_t^N) \leq K\varepsilon_{N}.
	\end{align*}
	Thus,  the other  desired assertion follows.  
\end{proof}

Next, we investigate the asymptotic convergence of the particle system and its empirical measure toward the unique invariant measure.

\begin{thm}
	Let  \textbf{(A1)} and \textbf{(A2)} hold with $p> (l/2)\vee 2 $. Then,  for any $2\leq q< p $,
	the solution $X^{i,N}_t$ and its empirical measure $u_t^N$ to IPSwCN \eqref{eq5.1} with initial data $(X^{1,N}_0, \cdots,X^{N,N}_0)= (\varsigma^1,\cdots, \varsigma^N)$  have the properties 
	\begin{align}
		\label{eq5.16}\max_{1\leq i\leq N}\lim_{t\rightarrow\infty}{W}_{q}^{q}\big(\mathcal{L}({X}^{i,N}_t),u^*\big)&\leq K\varepsilon_{N},
		~~~~~\lim_{t\rightarrow\infty}\mathbb{W}_{q}^{q}\big(\mathcal{L}( u_t^N),\mu^*\big)\leq K\varepsilon_{N},
	\end{align}
	where   $K>0$ is a constant independent of $N$. 
\end{thm}
\begin{proof}~~For any $1\leq i\leq N$, choose a random variable $\xi^i\in L(\Omega; \mathbb{R}^d)$ such that $\mathcal{L}(\xi^i)=u^*$ and $\mathcal{L}(\mathcal{L}^1(\xi^i))=\mu^*$. Let $\bar{X}^{i}_t$ be the solution to non-IPSwCN \eqref{eq5.2} with the initial value $\bar{X}^{i}_0=\xi^i$. Fix $2\leq q< p $. Since $(u^*,\mu^*)$ is the invariant measure couple of \eqref{eq5.2}, then we have for any $t\geq 0$,
	$${W}_q^q(\mathcal{L}(\bar{X}^{i}_t),u^*)\leq {W}_p^q(\mathcal{L}(\bar{X}^{i}_t),u^*)\leq d^q((\mathcal{L}(\bar{X}^{i}_t),\mathcal{L}(\mathcal{L}^1(\bar{X}^{i}_t))),(u^*,\mu^*))=0.$$ This together with Theorem \ref{T2} and \eqref{eq18} implies that 
	\begin{align*}
		\begin{aligned}
			{W}_q^q(\mathcal{L}(X^{i,N}_t),u^*)
			&\leq 3^{q-1}{W}_q^q(\mathcal{L}(X^{i,N}_t),\mathcal{L}(X^{i}_t))
			+3^{q-1}{W}_q^q(\mathcal{L}(X^{i}_t),\mathcal{L}(\bar{X}^{i}_t))
			\\&\quad+3^{q-1}{W}_q^q(\mathcal{L}(\bar{X}^{i}_t),u^*)
			\\&\leq 3^{q-1}\mathbb{E}|X_t^i-X_t^{i,N}|^q+3^{q-1}\mathbb{E}|X_t^i-\bar{X}_t^{i}|^q
			\\&\leq   K\varepsilon_{N}+3^{q-1}\mathbb{E}|\xi^i-\varsigma^i|^qe^{-\frac{q (c_3-c_4)t}{2}}.
		\end{aligned}
	\end{align*} 
	Letting $t\rightarrow\infty$ on both sides of the above inequality, we obtain the first inequality in \eqref{eq5.16} holds. Now we turn to the other one of \eqref{eq5.16}. Since $ \mu^* $ is the invariant measure  of the conditional distribution to \eqref{eq5.2},  we have for any $t\geq 0$,
	$$\mathbb{W}_{q}^{q}(\mathcal{L}(\mathcal{L}^1(\bar{X}^{i}_t)),\mu^*)\leq \mathbb{W}_{p}^{q}(\mathcal{L}(\mathcal{L}^1(\bar{X}^{i}_t)),\mu^*)\leq d^q((\mathcal{L}(\bar{X}^{i}_t),\mathcal{L}(\mathcal{L}^1(\bar{X}^{i}_t))),(u^*,\mu^*))=0.$$
	This together with \eqref{ls1} and \eqref{eq5.4} shows that
	\begin{align}\label{eq5.24}
		\begin{aligned}
			\mathbb{W}_{q}^{q}\big(\mathcal{L}( u_t^N),\mu^*\big)&\leq 3^{q-1}\mathbb{W}_{q}^{q}\big(\mathcal{L}( u_t^N),\mathcal{L}(\mathcal{L}^1(X_t^{i})) \big)
			\\
			&~~~+3^{q-1}\mathbb{W}_{q}^{q}\big(\mathcal{L}(\mathcal{L}^1(X_t^{i})),\mathcal{L}(\mathcal{L}^1(\bar{X}_t^{i}))\big)
			\\
			&~~~+3^{q-1}\mathbb{W}_{q}^{q}\big(\mathcal{L}(\mathcal{L}^1(\bar{{X}}_t^{i})),\mu^*\big)
			\\&\leq K\varepsilon_{N}+3^{q-1}\mathbb{E}|X_t^i-\bar{X}_t^i|^q.
	\end{aligned}\end{align}
	Using \eqref{eq18}, then letting $t\rightarrow\infty$ on both sides of \eqref{eq5.24} derives that
	\begin{align*}
		\lim_{t\rightarrow\infty}\mathbb{W}_{q}^{q}\big(\mathcal{L}( u_t^N),\mu^*\big)  
		\leq K\varepsilon_{N}+ 3^{q-1}\lim_{t\rightarrow\infty}\mathbb{E}|\xi^i-\varsigma^i|^qe^{-\frac{q (c_3-c_4)t}{2}}=K\varepsilon_{N}.
	\end{align*}
	The second inequality of \eqref{eq5.16} follows. The proof is complete.
\end{proof}
\section{Examples}\label{Sect.5}
We now illustrate our main results with two examples in different dimensions.
\begin{expl}
	\rm 
	Consider a scalar non-linear MV-SDEwCN
	\begin{align}\label{eq6.8}
		\begin{aligned}
			{\rm d}X_t&=(-X^3_t+(1-\beta_2)X_t+\beta_2\beta_3\mathbb{E}^1(X_t))\mathrm{d}t 
			+\sigma \mathrm{d}B_t+\beta_1 X_t\mathrm{d}B_t^0,
	\end{aligned}\end{align}
	where $\beta_1,\beta_2,\beta_3$ are constants.
\begin{figure}[htbp]\label{f1}
		\centering
		\includegraphics[height=6cm,width=13cm]{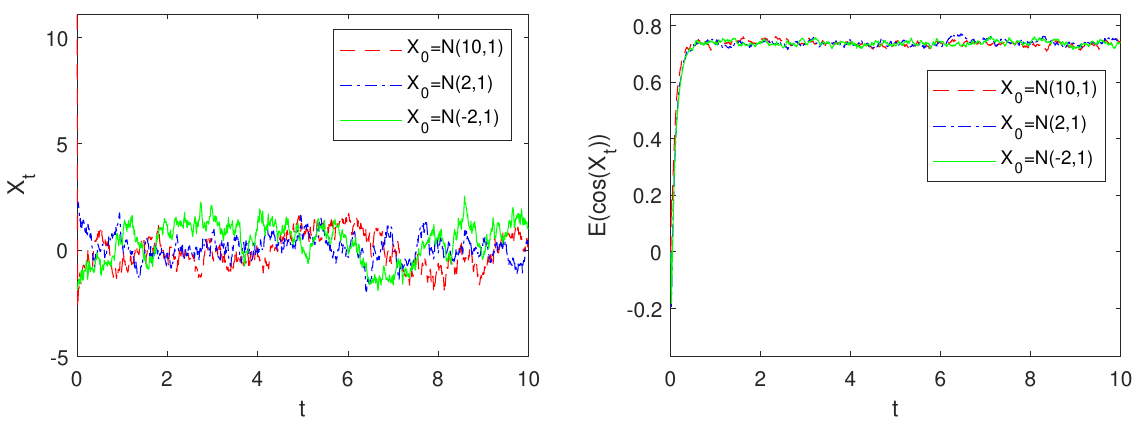}
		\setlength{\abovecaptionskip}{-0.03cm}
		\caption{Three sample paths of $X_t$ and the sample mean of  $\cos(X_t)$ with initial distributions $X_0\sim N(10,1),~N(2,1),~N(-2,1)$ of the system \eqref{eq6.8} for $t\in[0,10]$    with 1200 sample points.}
	\end{figure} When $\beta_1=0,\beta_3=1$, equation \eqref{eq6.8} is called the nonlinear anharmonic oscillator, refer to \cite{D1983}.  Comparing with \eqref{eq1.1}, we know
	$$
	f(x,u)=-x^3+(1-\beta_2)x+\beta_2\beta_3\int_{\mathbb{R}^d}zu(\mathrm{d}z),\quad g(x,u)=\sigma,\quad g^0(x,u)=\beta_1x. 
	$$ One notices that {\bf(A1)} holds with $c_1=3\big(1\vee(1-\beta_2)^2\vee(\beta_2\beta_3)^2\big), c_2=\sigma^2\vee\beta_1^2,l=6$. 
	We also have that 
	for any $x,y\in \mathbb{R}^d, u,v\in \mathcal{P}_2(\mathbb{R}^d)$, $\pi\in\Pi(u,v)$,
	\begin{align*}
		&(x-y)^T(f(x,u)-f(y,v))\nn\\
		&
		=(x-y)^T\left(-(x^3-y^3)+(1-\beta_2)(x-y)+\beta_2\beta_3\left(\int_{\mathbb{R}^d}zu(\mathrm{d}z)-\int_{\mathbb{R}^d}z^{\prime}v(\mathrm{d}z^{\prime})\right)\right)
		\nn\\&
		\leq (1-\beta_2)|x-y|^2+\beta_2\beta_3|x-y|\int_{\mathbb{R}^d}|z-z^{\prime}|\pi(\mathrm{d}z,\mathrm{d}z^{\prime})
		\nn\\&\leq(1-\beta_2+\frac{\beta_2\beta_3}{2})|x-y|^2+\frac{\beta_2\beta_3}{2}\int_{\mathbb{R}^d}|z-z^{\prime}|^2\pi(\mathrm{d}z,\mathrm{d}z^{\prime})
	\end{align*}
	Therefore, we have 
	\begin{align*}
		&(x-y)^T(f(x,u)-f(y,v))\leq(1-\beta_2+\frac{\beta_2\beta_3}{2})|x-y|^2+\frac{\beta_2\beta_3}{2}W_2^2(u,v).
	\end{align*}
	We compute that for any $\pi\in\Pi(u,v)$,
	\begin{align*}
		&|g(x,u)-g(y,v)|^2+|g^0(x,u)-g^0(y,v)|^2\leq \beta_1|x-y|^2 .
	\end{align*}
	Then we obtain that 
	\begin{align*}
		&2(x-y)^T(f(x,u)-f(y,v))+5|g(x,u)-g(y,v)|^2+5|g^0(x,u)-g^0(y,v)|^2\\
		&\leq \big(2(1-\beta_2)+5 \beta_1 ^2+\beta_2\beta_3\big)|x-y|^2+\beta_2\beta_3W_2^2(u,v).
	\end{align*}
	Then, if $2(1-\beta_2)+5 \beta_1^2+2\beta_2\beta_3<0$,  one observes that {\bf(A2)} holds with $c_3=-2(1-\beta_2)-5|\beta_1|^2-\beta_2\beta_3$, $c_4=\beta_2\beta_3$. Thus, by virtue of Theorem \ref{Th4.1},  MV-SDEwCN \eqref{eq6.8} has a unique invariant measure $(u^*,\mu^*)$. In particular, for $\beta_1=\frac{1}{2},\beta_2=3,\beta_3=\frac{1}{4},\sigma=2$, MV-SDEwCN \eqref{eq6.8} satisfies the conditions. 
Figure 1 depicts three sample paths of $X_t$ and the sample mean of $\cos(X_t)$ with initial data $X_0\sim N(10,1),N(2,1),N(-2,1)$ for $t\in[0,10]$ with 1200 sample points. Intuitively, although the solution processes start  from different initial data, they all  exhibit stationary characters in the sense of measure. Figure 2 depicts the empirical density of the solution process to MV-SDEwCN \eqref{eq6.8} with  $X_0\sim N(10,1)$ at time $t=20, 21, 22.5, 24, 24.9, 25$ with 1200 sample points. From the figure 2 one notices that the probability density of the solution process approaches to a steady state as time evolves.
\begin{figure}[htbp]\label{f11}
		\centering
		\includegraphics[height=6cm,width=12cm]{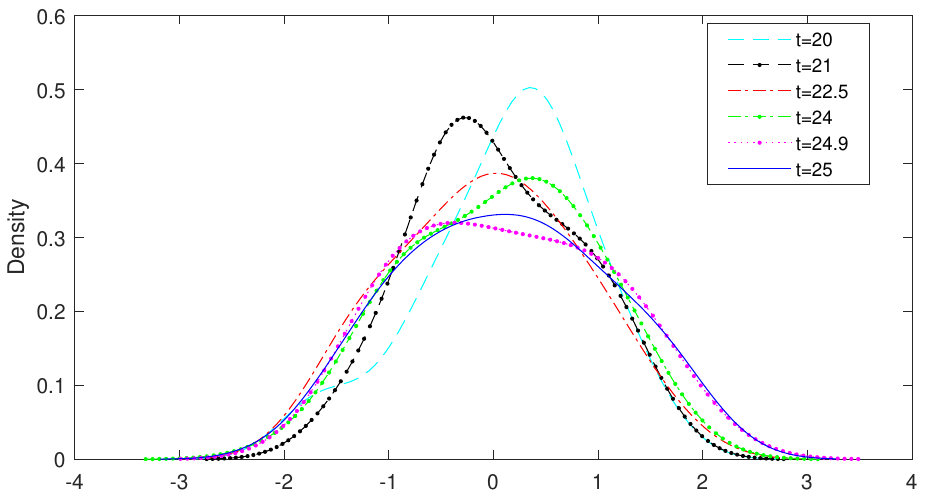}
		\setlength{\abovecaptionskip}{-0.03cm}
		\caption{The empirical density of the solution process to \eqref{eq6.8} with initial value $X_0\sim N(10,1)$ at time $t=20,21,22.5,24,24.9,25$ with 1200 sample points.}
	\end{figure}
\end{expl}
\begin{expl}
	\rm	Consider a two-dimensional MV-SDEswCN \eqref{eq1.1} with
	\begin{align*}
		&f(x, u)=\binom{-2 x_1-x_1^3+\frac{1}{2}\int x_1 u(\mathrm{d} x)}{-2 x_2+x_1-x_2^3+\frac{1}{2}\int x_2 u(\mathrm{d} x)}, \quad g(x, u)=\frac{1}{2}\left(\begin{array}{cc}
			\left(1-x_1\right) & 0 \\
			0 & \left(1-x_2\right)
		\end{array}\right),\end{align*} and
	\begin{align*}
		g^0(x, u)=\frac{1}{4}\left(\begin{array}{cc}
			\left(-x_1+\int x_1 u(\mathrm{d} x)\right) & 0 \\
			0 & \left(-x_2+\int x_2 u(\mathrm{d} x)\right)
		\end{array}\right),
	\end{align*}
	where $x=\left(x_1, x_2\right) \in \mathbb{R}^2$. It is easy to see that for any $x\in\mathbb{R}^2$, $u\in\mathcal{P}_2(\mathbb{R}^2)$ and $\pi\in\Pi(u,\delta_{\bf 0})$,
	$$
	|f(x, u)|^2
	\leq 16|x|^2+4|x|^6+\int |x|^2 \pi(\mathrm{d} x,\mathrm{d}y).
	$$
	The arbitrary of the coupling $\pi\in\Pi(u,\delta_{\bf 0})$ shows that 
	$
	|f(x, u)|^2\leq 16(|x|^2+|x|^6+W_2^2(u,\delta_{\bf 0})).
	$
	Similarly, we compute that
	$|g(x, u)|^2+|g^0(x, u)|^2\leq 1+x^2+W_2^2(u,\delta_{\bf 0})$.
	Thus {\bf(A1)} holds with $l=6,c_1=16,c_2=1.$
	A direct computation implies that for any $x,y\in \mathbb{R}^2$, $u,v\in\mathcal{P}_2(\mathbb{R}^d)$ and $\pi\in\Pi(u,v)$
	\begin{align*}
		&2\langle x-y, f(x,u)-f(y,v)\rangle +3(|g(x,u)-g(y,v)|^2+|g^0(x,u)-g^0(y,v)|^2)
		\\&\leq 
		-\frac{11}{8}|x-y|^2+\frac{5}{8}W_2^2(u,v).
	\end{align*}
	This implies that {\bf(A2)} holds with $p=4, c_3={11}/{8}$ and $c_4={5}/{8}$. By Theorem \ref{Th4.1}, MV-SDEswCN \eqref{eq1.1} has a unique invariant measure $(u^*,\mu^*)$. Figure 3 depicts the empirical density of the solution process to MV-SDEwCN \eqref{eq1.1} with  $X_0\sim N(10,1)$ at time $t=10$ with 500 sample points.
\begin{figure}[htbp]\label{f2}
		\centering
		\includegraphics[height=8cm,width=14cm]{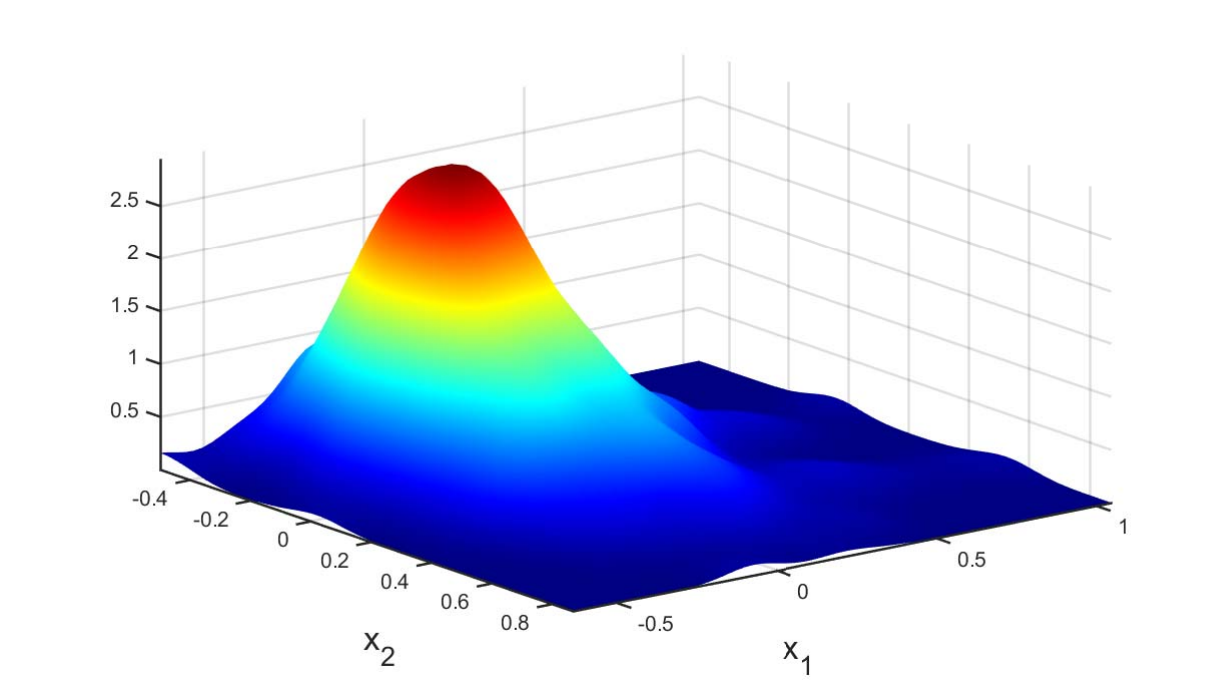}
		\setlength{\abovecaptionskip}{-0.03cm}
		\caption{The empirical density of the solution process to the system \eqref{eq1.1} with $X_0\sim N(10,1)$ at time $t=10$ with 500 sample points.}
	\end{figure}
\end{expl}
\section*{Acknowledgements}

The authors would like to thank the National Natural Science Foundation of China (No. 12371402)  
and the Tianjin Natural Science Foundation (24JCZDJC00830)
 for their financial support.

\bigskip
\noindent


\begin{thebibliography}{99}


\bibitem{ABG2012}Arapostathis A., Borkar V. S. and Ghosh M. K. (2012). \textit{Ergodic Control of Diffusion Processes}. Cambridge: Cambridge University Press.
	
	\bibitem{BSW2023}Bao J., Shao J. and Wei D. (2023). Wellposedness of conditional McKean-Vlasov equations with
	singular drifts and regime-switching. \textit{Discrete Contin. Dyn. Syst. Ser. B} \textbf{28} 2911--2926.
	
	\bibitem{BW}Bao J. and Wang J. (2024). Long time behavior of one-dimensional McKean-Vlasov SDEs with common noise. arXiv:2401.07665.
	
	\bibitem{BLW} Bao J., Liu Y. and Wang J. (2025). Ergodicity of conditional McKean-Vlasov jump diffusions. arXiv:2509.02249.
	
	\bibitem{BPP2023}Belomestny D., Pilipauskait$\mathrm{\dot{e}}$ V. and Podolskij M. (2023). Semiparametric estimation of McKean-Vlasov
	SDEs. \textit{Ann. Inst. Henri Poincar$\mathrm{\acute{e}}$ Probab. Stat.} \textbf{59} 79--96.
	
	\bibitem{BLY2022} Bo L., Li T. and Yu X. (2022). Centralized systemic risk control in the interbank system: weak formulation and Gamma-convergence. \textit{Stochastic Process. Appl.} \textbf{150} 622--654.
	
	\bibitem{BLM2023} Buckdahn R., Li J. and Ma J. (2023). A general conditional McKean-Vlasov stochastic differential equation. \textit{Ann. Appl. Probab.} \textbf{33} 2004--2023. 
	
	\bibitem{CD2018a} Carmona R. and Delarue F. (2018). \textit{Probabilistic Theory of Mean Field Games with Applications I: Mean Field FBSDEs, Control, and Games}. Cham: Springer.
	
	\bibitem{CD2018b} Carmona R. and Delarue F. (2018). \textit{Probabilistic Theory of Mean Field Games with Applications II: Mean Field Games with Common Noise and Master Equations}. Cham: Springer.   
	
	
	\bibitem{CD2022a} Chaintron L.-P. and Diez A. (2022). Propagation of chaos: a review of models, methods and applications. I. Models and methods. \textit{Kinet. Relat. Models} \textbf{15} 895--1015.
	%
	\bibitem{CD2022b} Chaintron L.-P. and Diez A. (2022). Propagation of chaos: a review of models, methods and applications. II. Applications. \textit{Kinet. Relat. Models} \textbf{15} 1017--1173.
	
	\bibitem{CST2022} Chassagneux J.-F., Szpruch L. and Tse A. (2022). Weak quantitative propagation of chaos via differential calculus on the space of measures. \textit{Ann. Appl. Probab.} \textbf{32} 1929--1969.
	
	\bibitem{CDHS2025}Chen C., Dang T., Hong J, and Song G. (2025). On numerical discretizations that preserve probabilistic limit behaviors for time-homogeneous Markov processes.
	\textit{Bernoulli} \textbf{31} 3139--3164.
	
	\bibitem{CF2016} Coghi M. and Flandoli F. (2016). Propagation of chaos for interacting particles subject to environmental noise. \textit{Ann. Appl. Probab.} \textbf{26} 1407--1442.
	
	\bibitem{CG2019} Coghi M. and Gess B. (2019). Stochastic nonlinear Fokker-Planck equations.  \textit{Nonlinear Anal.} \textbf{187} 259--278.
	
	\bibitem{DKL2014} Crisan D., Kurtz T. G. and Lee Y. (2014). Conditional distributions, exchangeable particle systems, and
	stochastic partial differential equations. \textit{Ann. Inst. Henri Poincaré Probab. Stat.} \textbf{50} 946--974. 
	
	\bibitem{DZ1996} Da Prato G.and Zabczyk J. (1996). \textit{Ergodicity for Infinite Dimensional Systems}. Cambridge: University Press.
	
	\bibitem{D1975} Dawson D. A. (1975). Stochastic evolution equations and related measure processes. \textit{J. Multivariate Anal.} \textbf{5} 1--52.
	
	\bibitem{D1983} Dawson D. A. (1983). Critical dynamics and fluctuations for amean-field model of cooperative behavior. \textit{J. Statist. Phys.} \textbf{31} 29--85.
	
	\bibitem{DTM} Delarue F., Tanr$\mathrm{\acute{e}}$ E. and Maillet R. (2024). Ergodicity of some stochastic Fokker-Planck equations with additive common noise. arXiv.2405.09950.
	
	\bibitem{DLW2024} Du K., Liu R. and Wang Y. (2024). Entropy solutions to the Dirichlet problem for nonlinear diffusion
	equations with conservative noise. \textit{SIAM J. Math. Anal.} \textbf{56} 637--675.
	
	\bibitem{ELL2021} Erny X., L\"ocherbach E. and Loukianova D. (2021). Conditional propagation of chaos for mean field systems of interacting neurons. \textit{Electron. J. Probab.} \textbf{26} Paper No. 20.
	
	\bibitem{FG2015} Fournier N. and Guillin A. (2015). On the rate of convergence in Wasserstein distance of the empirical measure. \textit{Probab. Theory Related Fields} \textbf{162} 707--738.
	
	\bibitem{GK1980} Gy\"ongy I. and Krylov N. V. (1980). On stochastic equations with respect to semimartingales. I. \textit{Stochastics} \textbf{4}  1--21.    
	
	\bibitem{HSS2021a} Hammersley W. R. P., $\mathrm{\check{S}}$i$\mathrm{\check{s}}$ka D. and Szpruch L. (2021). McKean-Vlasov SDEs under measure dependent Lyapunov conditions. \textit{Ann. Inst. Henri Poincaré Probab. Stat.} \textbf{57} 1032--1057.
	
	\bibitem{HSS2021} Hammersley W. R. P., $\mathrm{\check{S}}$i$\mathrm{\check{s}}$ka D. and Szpruch L. (2021). Weak existence and uniqueness for McKean-Vlasov SDEs with common noise. \textit{Ann. Probab.} \textbf{49} 527--555. 
	
	\bibitem{H2025} Huang X. (2025). Coupling by change of measure for conditional McKean-Vlasov SDEs and
	applications. \textit{Stochastic Process. Appl.} \textbf{179} Paper No. 104508.
	
	\bibitem{KW2012}Komorowski, T. and Walczuk, A. (2012). Central limit theorem for Markov processes with spectral gap in the Wasserstein metric. \textit{Stochastic Process. Appl.} \textbf{122} 2155–2184.
	
	\bibitem{K2018}Kulik, A. (2018). \textit{Ergodic Behavior of Markov Processes: With Applications to Limit Theorems.} De Gruyter Studies in Mathematics 67. Berlin: de Gruyter.
	
	\bibitem{KNRS2022} Kumar C., Neelima, Reisinger C. and Stockinger W. (2022). Well-posedness and tamed schemes for McKean-Vlasov equations with common noise. \textit{Ann. Appl. Probab.} \textbf{32} 3283--3330.
	
	\bibitem{LSZ2023}Lacker D., Shkolnikov M. and Zhang J. (2023). Superposition and mimicking theorems for conditional McKean-Vlasov equations. \textit{J. Eur. Math. Soc.} \textbf{25} 3229--3288.
	
    \bibitem{LXY2019} Li X., Mao X and Yin G. (2019). Explicit numerical approximations for stochastic differential equations in finite and infinite horizons: truncation methods, convergence in $p$th moment and stability, \textit{IMA J. Numer. Anal.}, \textbf{39} 847--892.
	
	\bibitem{LMW2021} Liang M., Majka M. B. and Wang J. (2021). Exponential ergodicity for SDEs and McKean-Vlasov processes with L$\mathrm{\acute{e}}$vy noise. \textit{Ann. Inst. Henri Poincaré Probab. Stat.} \textbf{57} 1665--1701.
	
	\bibitem{M}Maillet R. (2023). A note on the long-time behaviour of stochastic McKean-Vlasov equations with common noise. arXiv:2306.16130.
	
	\bibitem{M2008} Mao X. (2008). \textit{Stochastic Differential Equations and Applications, 2nd ed.} England: Horwood.
	
	\bibitem{M1966} McKean H. P. (1966). A class of Markov processes associated with nonlinear parabolic equations. \textit{Proc. Nat. Acad. Sci. U.S.A.} \textbf{56} 1907--1911.
	
	\bibitem{V2009} Villani C. (2009). \textit{Optimal transport, Old and new.} Berlin: Springer.  
	
	\bibitem{S2006}Shirikyan, A. (2006). Law of large numbers and central limit theorem for randomly forced PDE’s. \textit{Probab. Theory Related Fields} \textbf{134} 215--247.  
	
	\bibitem{S1991} Sznitman A. S. (1991). \textit{Topics in Propagation of Chaos.} Berlin: Springer.
	
	\bibitem{T2013} Tugaut J. (2013). Convergence to the equilibria for self-stabilizing processes in double-well
	landscape. \textit{Ann. Probab.} \textbf{41} 1427--1460.
	
	\bibitem{VHP2022}Vuong Y. V., Hauray M. and Pardoux $\mathrm{\acute{E}}$. (2022). Conditional propagation of chaos in a spatial stochastic epidemic model with common noise. \textit{Stoch. Partial Differ. Equ. Anal. Comput.} \textbf{10} 1180--1210.
	
	\bibitem{W2018}Wang F.-Y. (2018). Distribution dependent SDEs for Landau type equations. \textit{Stochastic Process. Appl.} \textbf{128} 595--621.
	
	\bibitem{W2021}Wang F.-Y. (2021). Image-dependent conditional McKean-Vlasov SDEs for measure-valued diffusion processes. \textit{J. Evol. Equ.} \textbf{21} 2009--2045.

\end{thebibliography}
\end{document}